\documentclass{siamart1116}
\usepackage{multirow}
\usepackage{graphicx}
\usepackage{epsfig}
\usepackage{amssymb}
\usepackage{enumerate}
\usepackage{amsmath}
\usepackage{amsfonts}
\usepackage{amsbsy}
\usepackage{mathrsfs}
\usepackage{setspace}
\usepackage{xcolor}

\newcommand{\Lo}{\mathcal{L}}
\newcommand{\C}{\mathcal{C}}
\newcommand{\T}{\mathscr{T}}

\newcommand{\m}{\mathcal{M}}

\newcommand{\ct}{\tilde{c}}

\newcommand{\wb}{\bar w}
\newcommand{\zb}{\bar z}

\newcommand{\qb}{\bar q}

\newcommand{\yb}{\bar y}
\newcommand{\zt}{\tilde z}
\newcommand{\rt}{\tilde \rho}
\newcommand{\pt}{\bar p}

\newcommand{\jh}{\hat{J}}
\newcommand{\lt}{\tilde L}

\newcommand{\tc}{\tilde \C}

\newcommand{\di}{\mathrm{div}\,}

\newcommand{\bnorm}[1]{|\kern-0.12em|\kern-0.12em|#1|\kern-0.12em|\kern-0.12em|}
\newcommand{\R}{\mathbb{R}}

\newcommand{\h}{\mathbf{H}}
\newcommand{\Lb}{\mathbf{L}}
\newcommand{\tih}{\tilde{\mathbf{H}}}

\newcommand{\junk}[1]{{}}

\newtheorem{rem}{Remark}
\newtheorem{cor}{Corollary}

\title{Multigrid preconditioners for the Newton-Krylov method in the optimal 
       control of the stationary {N}avier-{S}tokes equations 
       \thanks{This material is based upon work supported by the     
       U.S. Department of Energy Office of Science, Office of Advanced 
       Scientific Computing Research, Applied
       Mathematics program under Award Number DE-SC0005455,
       and by the National Science Foundation under awards
       DMS-1016177 and DMS-0821311.} 
      }

\author{Ana Maria Soane\thanks{Department of Mathematics,
        United States Naval Academy, 121 Blake Road Annapolis, Maryland 21402
        ({\tt soane@usna.edu}).} \and 
        Andrei Dr{\u a}g{\u a}nescu
        \thanks{Department of Mathematics and 
        Statistics, University of Maryland, Baltimore County, 
        1000~Hilltop Circle, Baltimore, Maryland 21250 ({\tt draga@umbc.edu}).
        }
        }
              
\begin{document}

\maketitle

\begin{abstract}
\textcolor{black}{The focus of this work is on the construction and analysis 
of  optimal-order multigrid preconditioners to be used in the 
Newton-Krylov method for a distributed optimal control problem constrained by 
the stationary Navier-Stokes equations. As in our earlier work~\cite{DS}
on the optimal control of the stationary Stokes equations, the strategy is to 
eliminate the state and adjoint variables from the optimality system and 
solve the reduced nonlinear system in the control variables. While 
the construction of the preconditioners extends naturally  the work in~\cite{DS},
the analysis shown in this paper 
presents a set of significant challenges that are rooted in the nonlinearity of 
the constraints. We also include numerical results that showcase the behavior of the proposed
preconditioners and show that for low to moderate Reynolds numbers they can
lead to significant drops in number of iterations and wall-clock savings.
}
\end{abstract}



\pagestyle{myheadings}
\thispagestyle{plain}
\markboth{A.~M.~SOANE AND A.~DR{\u A}G{\u A}NESCU}{MULTIGRID PRECONDITIONING FOR NAVIER-STOKES CONTROL}

\

\section{Introduction}

We consider the optimal control problem
\begin{align}\label{equ:cost}
 \min_{y,p,u} J(y,p,u) = 
 \frac{\gamma_y}{2}\|y - y_d\|^2_{\mathbf{L}^2(\Omega)} +
 \frac{\gamma_p}{2}\|p-p_d\|^2_{L^2(\Omega)} + 
 \frac{\beta}{2}\|u\|^2_{\mathbf{L}^2(\Omega)},
\end{align}
subject to the stationary Navier-Stokes equations
\begin{equation}\label{equ:ns}
\begin{aligned}
 -\nu \Delta y + (y \cdot \nabla)y + \nabla p 
               &= u \quad \text{in } \Omega, \\
         \di y &= 0 \quad \text{in } \Omega, \\
             y &= 0 \quad \text{on } \partial \Omega,
\end{aligned}
\end{equation}
where $\Omega \subset \R^2$ is a  bounded  convex polygonal domain. 
The goal of the control problem is to find a force $u$ that gives rise 
to a velocity $y$ and/or pressure $p$ to match a known target velocity 
$y_d$, respectively pressure $p_d$. Since this problem is ill-posed, we
consider a standard Tikhonov regularization for the force, with the
regularization parameter $\beta$ being a fixed positive number. 
The constants $\gamma_y, \gamma_p$ are nonnegative, not both zero.

\textcolor{black}{Model problems like~\eqref{equ:cost}--\eqref{equ:ns}
are commonly encountered in the literature on optimal control of partial
differential equations (PDEs), where boundary conditions, forcing terms, initial values, or coefficients
are treated as controls in order for the solution they determine to be close to a 
target state. For a given PDE-constraint, the associated  control problem that is most 
studied  is the case of distributed body forcing as control 
(e.g, see~\cite{MR2583281}).}

\textcolor{black}{Several works are centered on optimal control problems constrained by the Navier-Stokes equations, see e.g. 
\cite{GHS,GMnum, GMoc, MR1767769, MR1812735,  MR1858000, MR1905904, MR2179484, LRT} 
and the references therein, where both optimality conditions and numerical methods 
are addressed, for the unconstrained, control-constrained, or mixed 
control-state constrained problems. For a comprehensive overview of optimal flow control we refer the reader 
to \cite{Gb2}. In light of the potentially very large scale of the problems involved, a critical issue for all PDE-constrained optimization problems is to devise efficient solvers.
These solvers largely fall into two categories: 
the first kind targets the sparse but indefinite  Karush-Kuhn-Tucker (KKT) systems~\cite{RW, KZ},
while the second kind is centered on reduced systems. Our strategy falls in the second category.
More precisely we focus} on the efficient solution of the linear systems arising 
in the solution process of~\eqref{equ:cost}--\eqref{equ:ns}, specifically on 
the design of  multigrid preconditioners for the reduced Hessian in the 
Newton-CG method.  
To the best of our knowledge, this  has not been addressed in the literature 
for the Navier-Stokes optimal control problem.

\textcolor{black}{The multigrid preconditioning technique in this paper 
is rooted in the two- and multilevel methods for linear inverse problems proposed 
by Rieder~\cite{MR97k:65299}, Hanke and Vogel~\cite{MR2001h:65069}, and 
Dr{\u a}g{\u a}nescu and Dupont~\cite{DD}, the latter being primarily concerned with
regularized time-reversal of parabolic equations. The method has since been extended
to distributed control of linear elliptic equations 
for problems with control constraints~\cite{MR2888316, MR3175516, MR3537010}, 
distributed optimal control of semilinear elliptic equations~\cite{jyotithesis}, 
distributed optimal control of linear parabolic equations~\cite{monathesis}, 
as well as boundary control of elliptic equations~\cite{monathesis}.}

\textcolor{black}{The research in this article extends our  earlier work on 
the distributed optimal control of the Stokes equations~\cite{DS}; 
essentially, we show that for low to moderate Reynolds numbers 
the constructed preconditioners display the same optimal behavior as in the 
case of the Stokes-constrained problem. The fundamental 
departure from~\cite{DS}  resides, of course, in 
the nonlinearity of the constraints. Due to the linearity of the Stokes
equations in~\cite{DS}, the cost functional
of the reduced system is quadratic; thus the Hessian operator, and hence the 
preconditioner, is independent of the control. Instead, for the 
problem~\eqref{equ:cost}--\eqref{equ:ns}, the Hessian depends on the control, and
the preconditioner changes at every Newton iteration accordingly. While the construction of the preconditioner is
a natural extension of the one in~\cite{DS},
the main contribution  in this paper lies in the analysis.
The key element of the analysis
is the estimation of the $\mathbf{L}^2$-operator norm of the difference between the 
discrete Hessian and the two-grid preconditioner, as expressed in the main results of this
paper, Theorems~\ref{thm:hessv} and~\ref{thm:hessvp}. 
Due to the nonlinearity of the constraints, 
the discrete and continuous Hessians of the reduced 
cost functionals for~\eqref{equ:cost}--\eqref{equ:ns}
are more involved  than in~\cite{DS}, and hence the necessary error estimates leading to 
the aforementioned results are more challenging.
By comparison, the transition from linear elliptic 
 to semilinear elliptic constraints~\cite{jyotithesis}
was facilitated by the existence of $L^{\infty}$-estimates for the control-to-state
operator. The merit of our analysis is that we were able to avoid
$\mathbf{L}^{\infty}$ and  $\mathbf{W}_1^{\infty}$-estimates for the Navier-Stokes 
equations and its linearization which are more restrictive~\cite{MR3422453}.}
\textcolor{black}{For convenience, the analysis is restricted to  two dimensional domains, though most of
it can be extended to three dimensions with some restrictions on the discretization, 
at least for the case of velocity control ($\gamma_p=0$). 
The preconditioning formulation 
can be used without change for three-dimensional problems, since it is based on a
velocity-pressure formulation.}


The paper is organized as follows. In Section 2, we introduce the 
\textcolor{black}{reduced} optimal 
control problem and review \textcolor{black}{a set of} 
results that will be needed in the sequel. 
In Section 3, we introduce the discrete optimal control problem and \textcolor{black}{discuss}
finite element estimates that will be needed for the multigrid analysis. 
Section 4 contains {the} analysis of the two-grid preconditioner and the  main results of the paper. 
In Section 5, we showcase numerical experiments that 
illustrate our theoretical results. Conclusions \textcolor{black}{and 
a discussion of possible extensions} are presented in Section~6.

\section{Problem formulation}
\subsection{Preliminaries}
In this section we introduce  notations and review some classical 
\textcolor{black}{existence, uniqueness, and regularity}
results regarding the Navier-Stokes equations
\textcolor{black}{which will allow to formulate the reduced form~\eqref{equ:reduced}
of~\eqref{equ:cost}--\eqref{equ:ns}, and will play an essential role in the analysis}.  
We use standard notation for the Sobolev spaces $H^m(\Omega)$ and for their 
vector-valued counterparts we use the 
boldface notation. We denote by  $\tih^{-m}(\Omega)$  the dual (with respect 
to $\mathbf{L}^2$-inner product) of $\h^m(\Omega)\cap \h_0^1(\Omega)$  
and define $Q=L_0^2(\Omega)=\{p \in L^2(\Omega): \int_{\Omega} p \,dx = 0\}$, 
$X=\h_0^1(\Omega)$, and  
$V=\{v \in \h^1_0(\Omega):(\di v, q) = 0, \  \forall q \in Q\}$.
Throughout this paper  we write $(\cdot, \cdot)$ for the inner product in 
$L^2(\Omega)$ or $\mathbf{L}^2(\Omega)$, according to context, if there is 
no risk of misunderstanding. The $\h^m(\Omega)$ or $H^m(\Omega)$-norm
will be denoted by $\|\cdot\|_m$, while $\|\cdot\|$ denotes the
$\mathbf{L}^2(\Omega)$ or $L^2(\Omega)$-norm. Furthermore, define the norm in $V'$ by
$$\|u\|_{V'} = \sup_{\phi\in V\setminus\{0\}} (u,\phi) / \|\nabla \phi\|.$$

To define the weak formulation of~\eqref{equ:ns}, we introduce 
the 
bilinear forms 
\begin{align}
 a(y,\phi) &=  \nu (\nabla y,\nabla \phi) 
         =  \nu\sum_{i=1}^2\int_{\Omega} \nabla y_i\cdot \nabla \phi_i \,dx  
		    \quad \forall y,\phi \in  X, \\
 b(\phi,p) &= -\int_{\Omega} p \, \di \phi \, dx \quad \forall \phi \in X,  
            \forall p \in Q,
\end{align}
\textcolor{black}{as well as} the trilinear form
\begin{align}\label{equ:tri}
 c(y;\phi,\psi) = ((y\cdot \nabla)\phi,\psi)\quad 
 \forall y,\phi, \psi \in \h^1(\Omega).
\end{align}
A weak formulation of the Navier-Stokes equations is given by: 

\emph{Given $u \in \h^{-1}(\Omega)$, find 
$(y,p) \in X\times Q$ satisfying}
\begin{equation}\label{equ:nsweak}
\begin{aligned}
 a(y,\phi) + c(y;y,\phi) + b(\phi,p) &= 
           \langle u,\phi \rangle \quad \forall \phi \in X,\\
 b(y,q) &= 0 \quad \qquad \forall q \in Q,  
\end{aligned}
\end{equation}
where $\langle \cdot, \cdot \rangle$ denotes the duality pairing between 
$\h_0^1(\Omega)$ and $\h^{-1}(\Omega)$.  
Following~\cite{Lay}, the system~\eqref{equ:nsweak} can be written equivalently as:

\emph{Find $y\in V$ that satisfies }
\begin{equation}\label{equ:nsweakv}
 a(y,\phi) + c(y;y,\phi) = \langle u,\phi\rangle \quad \forall \phi \in V.
\end{equation}

We recall here a standard result regarding the existence of solution of 
\eqref{equ:nsweak} and uniqueness for small data, see e.g. \cite{GR,Lay}. 
For $\h^2$ regularity see \cite{LRricam}.
\begin{theorem}\label{thm:nse}
Let $\Omega \subset \R^2$ be  a bounded domain with Lipschitz continuous 
boundary.  Then for any $\nu >0$ and $u \in \h^{-1}(\Omega)$ there exists 
at least one solution $(y,p) \in V \times Q$ of the stationary Navier-Stokes 
problem \eqref{equ:nsweak} that satisfies the estimate 
\begin{align}\label{equ:ystab}
 \|\nabla y \| \leq \nu^{-1} \|u\|_{V'}. 
\end{align}
Moreover, the solution is unique if the data satisfies the smallness condition
\begin{align}\label{equ:unique}
 \m \nu^{-2} \|u\|_{V'} < 1,  \text{ with }  
 \m = \sup_{\phi,\psi,\chi \in X\setminus\{0\}}{ 
 \frac{|c(\phi;\psi,\chi)|}{\|\nabla \phi\| \|\nabla \psi\| \|\nabla \chi\|}}.
\end{align}
If $\Omega$ is convex and polygonal, and 
\mbox{$u \in \mathbf{L}^2(\Omega)$}, 
then $y \in \h^2(\Omega)$, $p\in H^1(\Omega)$ and
\begin{align}\label{equ:h2nse}
 \|y\|_{2} + \|p\|_{1} \leq 
 C(1+\|u\|^3).
\end{align}
\end{theorem}
\textcolor{black}{Recall that} throughout this paper we will assume $\Omega$ to be 
\textcolor{black}{a convex polygonal domain}, so that the 
$\h^2$-regularity of the Navier-Stokes problem is ensured. 
We state here some well-known results concerning the trilinear form defined in  
\eqref{equ:tri}, that will be needed in the sequel \cite{LR,GR,GMoc}.
\begin{lemma}\label{lemma:trilinear}
 The trilinear form  $c(y;\phi,\psi)$  defined in \eqref{equ:tri} 
 has the following properties: 
\begin{equation}\label{equ:trilinear}
 \begin{aligned}
  &c(y;\phi,\psi) = -c(y;\psi,\phi)
   \quad \forall y \in V, \forall \phi,\psi \in \h^1(\Omega),\\
  &c(y;\phi,\phi) =0 
   \quad \forall y \in V, \phi \in \h^1(\Omega),\\
  &c(y;\phi,\psi) = ((\nabla \phi)^T\psi,y)
   \quad \forall y,\phi,\psi \in \h^1(\Omega),\\
  &|c(y;\phi,\psi)| \leq \|y\|_1 \|\phi\|_1 \|\psi\|_1 
   \quad \forall y,\phi,\psi\in V,\\
  &|c(y;\phi,\psi)| \leq \m \|\nabla y \| \|\nabla \phi \| \| \nabla \psi\|
   \quad \forall y,\phi,\psi \in X, \\ 
  &|c(y;\phi,\psi)| \leq C\|u\|_1 \|\phi\|_1 \|\psi\|_1 
   \quad \forall y,\phi,\psi \in \mathbf{H}^1(\Omega), \\ 
  &
   |c(y;\phi,\psi)|\leq C \|y\| \|\phi\|_2\|\psi\|_1 
   \quad \forall y,\psi \in X, \phi \in \h^2(\Omega),
   \\
  &|c(y;\phi,\psi)|\leq C \|y\|_1 \|\phi\|_2\|\psi\| 
   \quad \forall y,\psi \in X, \phi \in \h^2(\Omega), \\
\end{aligned}
\end{equation}
with $\m$ given in \eqref{equ:unique} and $C$ independent of $y, \phi, \psi$.
\end{lemma}
\begin{proof}
While the others are standard, we prove here only the last estimate. Using H\"{older}'s inequality and the embedding 
 $\h^1(\Omega) \hookrightarrow \mathbf{L}^4(\Omega)$, we have
\begin{align*}
 |c(y;\phi,\psi)| = |((y\cdot \nabla)\phi,\psi)| \leq 
 \|y\|_{L^4(\Omega)}\|\nabla \phi\|_{L^4(\Omega)} \|\psi\|_{L^2(\Omega)}
 \leq C \|y\|_1\|\phi\|_2\|\psi\|.
\end{align*}
\end{proof}

When discretizing \eqref{equ:nsweak} using finite elements, in order 
to preserve the antisymmetry in the last two arguments of the trilinear form $c$ on the finite element 
spaces, it is standard to introduce a modified trilinear 
form  \cite{GMnum,Lay}  
\begin{align} \label{equ:ctilde}
 \ct(y;\phi,\psi) = \frac{1}{2}\{c(y;\phi,\psi)-c(y;\psi,\phi)\} 
                    \quad \forall y,\phi,\psi \in X, 
\end{align}
that has the following properties:
\begin{equation}
\begin{aligned}
  & c(y;\phi,\psi) = \ct(y;\phi,\psi) \quad 
                     \forall y \in V, \phi,\psi \in X, \\
  & \ct(y;\phi,\psi)   = -\ct(y;\psi,\phi) \quad \forall y,\phi,\psi \in X, \\
  & \ct(y;\psi,\psi)  = 0 \quad \forall y, \textcolor{black}{\psi} \in X, \\
  & |\ct(y;\phi,\psi)| \leq \m \|\nabla y \| \| \nabla \phi \| \|\nabla \psi \|
                       \quad \forall y,\phi,\psi \in X, 
\end{aligned}
\end{equation}
for the same $\m=\m(\Omega)$ as in \eqref{equ:unique}.
Thus, another variational formulation of \eqref{equ:nsweak} 
is:

\emph{Given $u\in \mathbf{H}^{-1}(\Omega)$, 
find $(y,p) \in X \times Q$ satisfying 
\begin{equation}\label{equ:nsweak3}
\begin{aligned}
 a(y,\phi) + \ct(y;y,\phi) +b(\phi,p) &= 
             \langle u, \phi \rangle \quad \forall \phi \in X,\\
 b(y,q) &= 0 \quad \qquad \forall q \in Q.
\end{aligned}
\end{equation}
}

We define the set of admissible controls $U =\{ u: \mathbf{L}^2(\Omega): \|u \| < \nu^2/(\m\kappa) \}$,
with $\m$ defined in \eqref{equ:unique} and $\kappa$ 
the embedding constant of $\mathbf{L}^2(\Omega)$ into $V'$. 
By Theorem~\ref{thm:nse}, 
the Navier-Stokes equations have a unique solution for each 
$u\in U$ on the right hand side of \eqref{equ:nsweak}.   
We introduce the control-to-state operators  
$Y:U \rightarrow V$, $P:U\rightarrow Q$  that assign to each  
$u \in U \subset \mathbf{L}^2(\Omega)$ the corresponding Navier-Stokes 
velocity $y = Y(u)$ and pressure $p=P(u)$, 
and rewrite problem \eqref{equ:cost} in reduced form as 
\begin{align}\label{equ:reduced}
 \min_{u\in U} \jh(u) = \frac{\gamma_y}{2}\|Y(u) - y_d \|^2 +
                        \frac{\gamma_p}{2}\|P(u) - p_d \|^2 +
                        \frac{\beta}{2} \|u\|^2. 
\end{align}
Throughout this paper we will assume that the target velocity field
$y_d$ is from $\h^{1}(\Omega)$; the target pressure $p_d$ is \textcolor{black}{assumed for now
to be in $Q$}.

We note that for all pairs $(y(u),u)$ with $u \in U$, we have
\begin{align}\label{equ:ell}
 \nu >\m(y), \quad \text{with } 
 \m(y) := \sup_{v \in X} \frac{|c(v;y,v)|}{\|\nabla v\|^2},
\end{align}
which 
ensures the ellipticity of the linearized equations about $y$. 
\textcolor{black}{The following estimate establishes the regularity of the 
solution of the linearized equations about $y$, and will be used in Section~\ref{sec:twogridsection}.}
\begin{lemma}\label{lemma:reglin}
Let $u \in U$ and $y=Y(u)\in V$. Then for every $g \in V'$ there exists 
a unique weak solution $(w,r) \in X\times Q$ of the linearized 
Navier-Stokes system
\begin{equation}\label{equ:lin0}
\begin{aligned}
 -\nu \Delta w + (w\cdot \nabla)y + (y\cdot\nabla) w + \nabla r 
       &= g  \quad \text{ in } \Omega,\\
 \di w &= 0  \quad \text{ in } \Omega,\\
     w &= 0 \quad \text{ on } \partial \Omega, 
\end{aligned}
\end{equation}
and 
\begin{align}\label{equ:estlinh1}
 \|\nabla w\| \leq \frac{2}{\nu}\|g\|_{V'}.
\end{align} 
If $\Omega$ is \textcolor{black}{convex} and polygonal, and 
$g\in \mathbf{L}^2(\Omega)$, then $w\in \h^2(\Omega)$, 
 \mbox{$r \in H^1(\Omega)$}, and 
\begin{align}\label{equ:estlinh2}
\|w\|_2 \leq C(y) \|g\|.
\end{align}
\end{lemma}
\begin{proof}
Existence and uniqueness follows from the Lax-Milgram lemma, 
using \eqref{equ:ell} to prove the ellipticity of the associated 
bilinear form. For the proof of \eqref{equ:estlinh1} see \cite{TW}, 
Corollary 3.7. 
To prove  \eqref{equ:estlinh2}, we note that for $g \in \mathbf{L}^2(\Omega)$, 
we have  
$(w\cdot \nabla) y$, $(y\cdot \nabla) w$  $\in \mathbf{L}^2(\Omega)$ (see estimates below); thus   
by rewriting  \eqref{equ:lin0} as 
\begin{align*} 
 -\nu \Delta w + \nabla r = g - (w\cdot \nabla) y - (y\cdot \nabla) w,
\end{align*}
we can use standard regularity results for the Stokes equations to obtain 
\begin{align}\label{equ:stokes}
 \|\nabla \nabla w \| 
 \leq C_1(\Omega)(\|g\| + \|(w\cdot \nabla)y\| + \|(y\cdot\nabla)w\|).
\end{align}
We have 
\begin{align*}
 \|(w\cdot \nabla)y\|^2 &= \int_{\Omega} |(w\cdot \nabla)y|^2 dx 
 \leq \int_{\Omega} |w|^2 |\nabla y|^2 dx 
  \leq \|w\|^2_{\mathbf{L}^4(\Omega)} \|\nabla y\|^2_{\mathbf{L}^4(\Omega)},
\end{align*}
which implies
\begin{align}\label{equ:h1l4}
 \|(w\cdot \nabla)y\| \leq C \|w\|_1 \|\nabla y\|_1
 \leq C_1(y) \|g\|,
\end{align}
since $\h^1(\Omega) \hookrightarrow \mathbf{L}^4(\Omega)$. Similarly, it can be shown 
that 
\begin{align*}
 \|(y\cdot \nabla) w\| \leq C \|y\|_1 \|\nabla w \|_{\mathbf{L}^4(\Omega)}\leq C_2(y) 
 \|\nabla w \|^{1/2} \|\nabla \nabla w\|^{1/2}, 
\end{align*}
where we used Ladyzhenskaya's inequality, 
$$
 \|\nabla w \|_{\mathbf{L}^4(\Omega)} \leq C 
 \|\nabla w \|^{1/2} \|\nabla \nabla w \|^{1/2}.
$$
Finally, using Young's inequality we obtain 
\begin{align*}
 \|(y\cdot \nabla)w\| &\leq C_2(y)\Big(
 \frac{1}{2}C_2(y)C_1(\Omega) \|\nabla w\| + 
 \frac{1}{2C_2(y)C_1(\Omega)} \|\nabla \nabla w\|\Big) \\
 &= 
 \frac{1}{2}C_2^2(y)C_1(\Omega) \|\nabla w\| + 
 \frac{1}{2C_1(\Omega)}\|\nabla\nabla w\|.
\end{align*}
Substituting  in \eqref{equ:stokes} gives
\begin{align*}
 \|\nabla \nabla w\| \leq C_1(\Omega)\Big(\|g\| + C_1(y) \|g\|+ 
 \frac{1}{2}C_2^2(y)C_1(\Omega)\|\nabla w\| + 
 \frac{1}{2C_1(\Omega)}\|\nabla \nabla w\|\Big),
\end{align*}
from which \eqref{equ:estlinh2} follows immediately.
\end{proof}
We recall here the following results from \cite{LRT} regarding the 
differentiability of the solution operators $Y$, $P$. 
\begin{theorem}
\label{th:twicediff}
Let $u \in U$ and $y = Y(u)$. The control-to-state operators $Y$, $P$ are  
twice Fr\'{e}chet differentiable at $u$ and their derivatives $w = Y'(u)g$, 
$r=P'(u)g$ and
$\lambda=Y''(u)[g_1,g_2]$, $\mu=P''(u)[g_1,g_2]$
are given by the unique weak solutions of the systems:
\begin{equation}\label{equ:lin}
\begin{aligned}
 -\nu \Delta w + (w\cdot \nabla)y + (y\cdot\nabla) w + \nabla r 
       &= g  \quad \text{ in } \Omega,\\
 \di w &= 0  \quad \text{ in } \Omega,\\
     w &= 0 \quad \text{ on } \partial \Omega, 
\end{aligned}
\end{equation}
and
\begin{equation}\label{eq:secder}
\begin{aligned}
     -\nu \Delta  \lambda + (y \cdot \nabla) \lambda 
     + (\lambda \cdot \nabla )y + \nabla \mu 
  &= -(Y'(u)g_1 \cdot \nabla)Y'(u)g_2\\
  &  \quad - (Y'(u)g_2\cdot \nabla) Y'(u)g_1  \quad \text{in } \Omega,\\
     \di \lambda &= 0 \hspace{38mm} \text{in } \Omega,\\
     \lambda &= 0 \hspace{38mm} \text{on } \partial \Omega,
\end{aligned}
\end{equation}
respectively.
\end{theorem}
\junk{
\begin{rem}
 Note that $P$ is also twice Fr\'{e}chet differentiable and 
 $r=P'(u)g$ and $\mu=P''(u)[g_1,g_2]$ in \eqref{equ:lin},  
 \eqref{eq:secder} respectively.
 added in the Lemma).  
\end{rem}
}
\begin{lemma}
 Let $u \in U$, $y = Y(u)$, and $Y'(u)^*$ be the adjoint of $Y'(u)$. 
 Then  $z = Y'(u)^* g$ is the first component of the 
unique weak solution $(z,\rho)$ of the system
\begin{equation}\label{equ:adjoint} 
\begin{aligned}
-\nu \Delta z - (y \cdot \nabla) z + (\nabla y)^T z + \nabla \rho 
       &= g  \quad \text{in } \Omega,\\
 \di z &= 0  \quad \text{in } \Omega,\\
     z &= 0 \quad \text{on }  \partial \Omega.
\end{aligned}
\end{equation}
If $\Omega \subset \R^2$ is a convex polygonal domain then 
$z \in \h^2(\Omega)$, $\rho \in H^1(\Omega)$ and
\begin{align} \label{equ:regalin}
 \|z\|_2 \leq C(y)\|g\|.
\end{align}
\end{lemma}
\begin{proof}
 See \cite[Theorem 3.10]{TW} and Lemma~\ref{lemma:reglin}.
\end{proof}

\subsection{Optimality conditions}
We derive next the first-order necessary optimality conditions associated 
with the optimal control problem \eqref{equ:reduced}. 
For $g \in \mathbf{L}^2(\Omega)$, 
\begin{align*}
 \jh'(u)g = \gamma_y(Y(u) -y_d, Y'(u) g) + 
            \gamma_p(P(u)-p_d, P'(u)g) + \beta(u,g),   
\end{align*}
therefore
\begin{align}\label{equ:grad}
 \nabla \jh(u) = \gamma_y Y'(u)^*(Y(u)-y_d) + 
                 \gamma_p P'(u)^*(P(u)-p_d) + \beta u.
\end{align}
Thus, the optimal control $u$ is the solution of the non-linear equation  
\begin{align}\label{equ:first}
 \gamma_y Y'(u)^*(Y(u)-y_d) + \gamma_p P'(u)^*(P(u)-p_d) + \beta u = 0.
\end{align}
The reduced Hessian is computed using the second variation of $\jh$: if
\mbox{$g_1,g_2 \in \mathbf{L}^2(\Omega)$}
\begin{equation}\label{jrv}
\begin{aligned}
 \jh''(u)[g_1,g_2] &= \gamma_y(Y'(u)g_2,Y'(u)g_1) + 
                      \gamma_y (Y(u)-y_d, Y''(u)[g_2,g_1]) \\
                   &+ \gamma_p(P'(u)g_2,P'(u)g_1) + 
				      \gamma_p(P(u)-p_d, P''(u)[g_2,g_1]) 
					+ \beta(g_1,g_2).
\end{aligned}
\end{equation}
We use different 
approaches in proving the main multigrid results, depending on whether 
the pressure term is present in the cost functional~\eqref{equ:cost} or not, therefore 
we will derive the reduced Hessian for the two cases separately. 
\junk{ 
We denote  by $L$ and $M$ the  solution operators of the 
equations 
\eqref{equ:lin}, such that $ L h = Y'(u) h $, $M h = P'(u)h$. 
Although $L$, $M$ depend on $y=y(u)$ in \eqref{equ:lin}, we use the 
notation $L$, $M$ instead of $L_y$, $M_y$ or $L(u)$, $M(u)$, for simplicity, if there 
is no risk of misunderstanding. 
Using 
$Y''(u)[g_1,g_2]=
L(-(L g_1\cdot\nabla)L g_2-(L g_2\cdot\nabla)L g_1)$ and 
$P''(u)[h_1,h_2]=
M(-(L g_1\cdot\nabla)L g_2-(L g_2\cdot\nabla)L g_1)$ 
in \eqref{jrv}, we obtain 
\begin{align*}
\jh''(u)[g_1,g_2] 
&= \gamma_y[(Lh_1,Lh_2)+((Lh_1\cdot \nabla)z,Lh_2)-((\nabla Lh_1)^T z,Lh_2)]
 \\ &+\gamma_p[(M g_1,M g_2)
+((Lh_1\cdot \nabla)M^*(p-p_d),Lh_2)\\
&-((\nabla Lh_1)^T M^*(p-p_d),Lh_2)]
+ \beta(h_1,h_2).
\end{align*}
Thus, the Hessian operator associated to the reduced cost functional, 
$\jh(u)$ is given by 
\begin{equation}\label{equ:hessc}
\begin{aligned}
    H_{\beta}(u)v 
 &= \beta v+\gamma_y[L^*Lv + L^*((Lv\cdot \nabla)-(\nabla Lv)^T)L^*(y-y_d)\\
 &+ \gamma_p[M^*M v + L^*((Lv\cdot \nabla)-(\nabla Lv)^T)M^*(p-p_d)].
\end{aligned}
\end{equation}
}  

\subsubsection{Velocity control only}

We consider first the case of velocity control
only, i.e., $\gamma_y=1,\gamma_p=0$. 
\junk{
We have, for every $h \in \mathbf{L}^2(\Omega)$,
\begin{align*}
 \jh'(u)g = (Y(u) -y_d, Y'(u) g) + \beta(u,g)   
\end{align*}
and
\begin{align}\label{equ:grad}
 \nabla \jh(u) = Y'(u)^*(Y(u)-y_d) + \beta u.
\end{align}
Thus, the optimal control $u$ is the solution of the non-linear equation
\begin{align}\label{equ:first}
  Y'(u)^*(Y(u)-y_d) + \beta u = 0.
\end{align}
}
In this case the  second variation of $\jh$
becomes 
\begin{equation}\label{equ:jr_v}
\begin{aligned}
 \jh''(u)[g_1,g_2] &= (Y'(u)g_2,Y'(u)g_1) + 
                      (Y(u)-y_d, Y''(u)[g_2,g_1]) + \beta(g_1,g_2).
\end{aligned}
\end{equation}
We denote  by $L$ and $M$ the  solution operators of 
\eqref{equ:lin}, such that $ L g = Y'(u) g$, $M g = P'(u)g$.
Although $L$, $M$ depend on $y=y(u)$ in \eqref{equ:lin}, we use the
notation $L$, $M$ instead of 
$L(u)$, $M(u)$, for simplicity, 
when there is no risk of misunderstanding. 
Cf. Theorem~\ref{th:twicediff}, $\lambda=Y''(u)[g_1,g_2]$ is the solution of 
\begin{equation}\label{sd}
\begin{aligned}
 a(\lambda,\phi)+c(y;\lambda,\phi) &+ c(\lambda;y,\phi)
 +b(\phi,\mu) \\
  &= -c(L g_1;L g_2;\phi)-c(L g_2;L g_1,\phi)
 \quad \forall \phi \in X, \\
 b(\lambda,q) &=0 \hspace*{5.2cm} \forall q \in Q.
\end{aligned}
\end{equation}
Similarly, we let $z=L^*(Y(u)-y_d)$. Note that is the solution of
\begin{equation}\label{ad}
\begin{aligned}
 a(z,\phi) + c(y;\phi,z) + c(\phi;y,z) + b(\phi,\rho) &= (y-y_d,\phi) 
 \quad \forall \phi \in X, \\
 b(z,q) &= 0 \qquad \qquad \quad \forall q \in Q.
\end{aligned}
\end{equation}
By taking $\phi = z$ in \eqref{sd} and $\phi = \lambda$ in \eqref{ad} 
we obtain
\begin{align}
 -c(L g_1;L g_2;z)-c(L g_2;L g_1,z) = (Y(u)-y_d, \lambda).
\end{align}
Using this in \eqref{equ:jr_v} we get
\begin{align*}
 \jh''(u)[g_1,g_2] 
 &= (L g_1,L g_2)-c(L g_1;L g_2,z)-c(L g_2;L g_1,z) + \beta(g_1,g_2)\\ 
 &= (L g_1,L g_2)+((L g_1\cdot \nabla)z,L g_2)-((\nabla L g_1)^T z,L g_2)
               +\beta(g_1,g_2).
\end{align*}
The Hessian operator associated with $\jh$, defined by $(H_{\beta}(u)v,g)=\jh''(u)[v,g]$,
 is  
\begin{equation}\label{equ:hess_v}
\begin{aligned}
 H_{\beta}(u)v 
  =\beta v + L^*Lv + L^*((Lv\cdot \nabla) - (\nabla Lv)^T)L^*(y-y_d).
\end{aligned}
\end{equation}
To simplify the presentation we introduce the notation 
\begin{align*}
 A(u)v =L^*L v, \quad \C(u)v =  L^*((Lv\cdot \nabla) - (\nabla Lv)^T)L^*(Y(u)-y_d),
\end{align*}
that we will use throughout the paper. Note that
\begin{align}\label{equ:cform}
 (\C(u)v,v) = -2c(Lv;Lv,L^*(Y(u)-y_d)).
\end{align}

\subsubsection{Mixed/pressure control}
\label{ssec:Mixed_pressure_control}
Here we consider the general case of mixed velocity/pressure
control or pressure control only, i.e, $\gamma_p\neq0$. 

Let $(\zt,\rt)$ be the solution of the 
problem 
\begin{equation}\label{equ:adj}
\begin{aligned}
 a(\zt,\phi) + c(y;\phi,\zt) + c(\phi;y,\zt) + b(\phi,\rt) 
          &= \gamma_y(y-y_d,\phi) \quad \forall \phi \in X, \\
 b(\zt,q) &= \gamma_p(p-p_d,q) \quad \forall q \in Q,
\end{aligned}
\end{equation}
which is the weak form of the problem 
\begin{equation}\label{equ:adj_m}
\begin{aligned}
-\nu \Delta \zt - (y \cdot \nabla) \zt + (\nabla y)^T \zt + \nabla \rt 
       &= \gamma_y(y-y_d)  \quad \text{in } \Omega,\\
 \di \zt &= \gamma_p(p_d-p)  \quad \text{in } \Omega,\\
     \zt &= 0 \quad \qquad \qquad \text{on }  \partial \Omega.
\end{aligned}
\end{equation}
By taking $\phi=\lambda$ in \eqref{equ:adj}, $\phi =\zt$ in 
\eqref{sd}, and using $b(\lambda,\rt)=0$, 
$b(\zt,\mu)=\gamma_p(p-p_d,\mu)$ 
we obtain
\begin{align*}
 \gamma_y(y-y_d,\lambda) + \gamma_p(p-p_d,\mu) = 
 -c(L g_1;L g_2, \zt)-c(L g_2;L g_1, \zt).
\end{align*}
Thus, the second variation of the reduced cost functional \eqref{jrv} becomes
\begin{equation}
\begin{aligned}
   \jh''(u)[g_1,g_2] &= 
   \gamma_y(Y'(u)g_2,Y'(u)g_1)+\gamma_p(P'(u)g_2,P'(u)g_1) +\\
 &\hspace*{3mm}- c(L g_1;L g_2,\zt)-c(L g_2;L g_1,\zt)+
 \textcolor{black}{\beta(g_1,g_2)}
\end{aligned}
\end{equation}
and the reduced Hessian is 
given by
\begin{equation}\label{equ:rhvp}
\begin{aligned}
 H_{\beta}(u)v = \beta v + \gamma_y L^*Lv + \gamma_pM^*M v
                  + L^*((Lv\cdot\nabla)\zt - (\nabla L v)^T\zt).
\end{aligned}
\end{equation}
We introduce the notation
\begin{align}
\label{eq:ctilde}
 \tc(u)v = L^*((Lv\cdot\nabla)\zt - (\nabla L v)^T\zt) 
\end{align}
and note that
\begin{equation}
\label{equ:Ctilde}
 (\tc(u)v,v) = -2c(Lv;Lv,\zt).
\end{equation}
 Note that if we take $\gamma_y=1$, $\gamma_p=0$ in \eqref{equ:adj}, then 
 \eqref{equ:adj} is the adjoint linearized Navier-Stokes system 
 and 
 in this case \textcolor{black}{\eqref{equ:rhvp} reduces to \eqref{equ:hess_v}}.


\section{Discretization and approximation results}
\label{sec:discretization}
\textcolor{black}{
In order to discretize the optimization problem~~\eqref{equ:cost}--\eqref{equ:ns}
we adopt the strategy to} first discretize the Navier-Stokes system, then
optimize the cost functional $J$ in~\eqref{equ:cost} subject to the discrete constraints.

\subsection{Finite element approximation}

We consider a shape regular,  quasi-uniform quadrilateral mesh
$\T_h$ of $\bar{\Omega}$, and we assume that the mesh $\T_h$  results from
a coarser regular mesh $\T_{2h}$ from one uniform refinement. We use the
Taylor-Hood $\mathbf{Q}_2-\mathbf{Q}_1$ finite elements
to discretize the state equation.
The velocity field $y$ is approximated in the space
$X_h^0 = X_h \cap \h_0^1(\Omega)$, where
\begin{align*}
 X_h  &= \{v_h \in C(\bar{\Omega})^2: v_h|_T \in \mathbf{Q}_2(T)^2 
           \text{ for } T \in \T_h\}
\end{align*}
and the pressure $p$ is approximated in the space
\begin{align*}
 Q_h   &= \{q_h \in C(\Omega) \cap L^2_0(\Omega): q_h|_T \in \mathbf{Q}_1(T) 
			          \text{ for } T \in \T_h\},
\end{align*}
where $\mathbf{Q}_k(T)$ is the space of polynomials of degree less than or equal
to~$k$ in each variable~\cite{Ci}.
The control variable $u$ is approximated by continuous piecewise 
biquadratic polynomial vector functions from $X_h$. 
We also introduce the space
\begin{align}
 V_h = \{v_h \in X_h^0: (\di v_h, q_h) = 0 \ \forall q_h \in Q_h\}
\end{align}
and note that $V_h \nsubseteq V$.

\begin{rem}
The choice  to work with quadrilateral $\mathbf{Q}_2-\mathbf{Q}_1$
Taylor-Hood elements was made for convenience and clarity of exposition; 
our analysis can be extended to triangular
$\mathbf{P}_2-\mathbf{P}_1$ elements as well as other stable mixed
finite elements.
\end{rem}

For a given control $u_h \in X_h \cap U$, the solution $(y_h,p_h)$ of 
the discrete state equation is given by
\begin{equation}\label{equ:nsd}
\begin{aligned}
 a(y_h, \phi_h) + \ct(y_h;y_h,\phi_h) + b(\phi_h,p_h) 
             &= (u_h, \phi_h) \quad \forall \phi_h \in X_h^0, \\
 b(y_h, q_h) &= 0        \qquad \qquad \ \forall q_h \in Q_h. 
\end{aligned}
\end{equation}
Let $Y_h$ and $P_h$ be the solution mappings of the discretized state 
equation, defined analogously to their continuous counterparts. 
The discretized, reduced optimal control problem reads
\begin{align}\label{equ:ocd}
 \min_{u_h} \jh_h(u_h) = \frac{\gamma_y}{2}\|Y_h(u_h) - y_d^h\|^2 +
                         \frac{\gamma_p}{2}\|P_h(u_h) - p_d^h\|^2 +
                         \frac{\beta}{2}\|u_h\|^2,
\end{align}
where $y_d^h,p_d^h$ are the $L^2$-projections of the data onto $X_h$,
respectively $Q_h$.

We denote by  $L_h$, $M_h$ the solution operators  of the discretized 
linearized Navier-Stokes equations (about $y_h$), i.e., $L_h g =w_h$, $M_h g=r_h$, where
\begin{equation}\label{equ:lnsed}
\begin{aligned}
 a(w_h, \phi_h) &+\ct(y_h;w_h,\phi_h) + 
     \ct(w_h;y_h,\phi_h) + b(\phi_h,r_h)  \\ 
	  &= (g,\phi_h) 
	    \quad \forall \phi_h \in X_h^0, \\
		   b(w_h, q_h) &= 0 \quad \forall q_h \in Q_h,
\end{aligned}
\end{equation}
We remark that, as in the continuous case, $z_h=L_h^*g$ satisfies
\begin{equation}\label{equ:alnsed}
\begin{aligned}
 a(z_h,\phi_h) &+\ct(y_h;\phi_h,z_h)+\ct(\phi_h;y_h,z_h) + b(\phi_h,\rho_h)\\ 
	       &= (g,\phi_h)  \quad \forall \phi_h \in X_h^0,\\
    b(z_h, q_h)&= 0 \quad \forall q_h \in Q_h.
\end{aligned}
\end{equation}

\subsection{A priori estimates}
\textcolor{black}{
In this section we collect several approximation results pertaining to the
finite element approximation of the Navier-Stokes equations and the
linearized/adjoint linearized Navier-Stokes equations that will be needed 
for the multigrid analysis. 
}
\begin{lemma}\label{lemma:proj}
  Let $\pi_h$ be the $L^2$-orthogonal projection onto $X_h$. 
  The following approximation properties hold:
 \begin{align}\label{equ:idpr}
 \|(I-\pi_h)v\|_{\tih^{-k}(\Omega)} \leq C h^k \|v\| \quad \forall 
  v \in \mathbf{L}^2(\Omega),\ \ k=1,2,
 \end{align}
 \begin{align}\label{equ:neg}
  \|(I-\pi_{h}) u\|_{\tih^{-1}(\Omega)} \leq C h^2 \|u\|_1
  \quad \forall u \in \h^1(\Omega), 
 \end{align}
with $C$ independent of $h$.
\end{lemma}
\begin{proof}
  The estimate~\eqref{equ:idpr} is a  standard result (e.g., see~\cite{DD}). 
  For~\eqref{equ:neg}, let $I_{h}:\h^1(\Omega) \rightarrow X_{h}$ be the interpolant 
introduced by Scott and Zhang in \cite{SZ}. 
We have 
\begin{align*}
 \|u-\pi_{h} u\|_{\tih^{-1}(\Omega)} 
 &=    \sup_{v \in \h^1_0(\Omega)\setminus\{0\}}\frac{(u-\pi_{h} u,v)}{\|v\|_1}
  =    \sup_{v \in \h^1_0(\Omega)\setminus\{0\}}
       \frac{(u-\pi_{h} u,v-I_{h} v)}{\|v\|_1}\\
 &\leq \sup_{v \in \h^1_0(\Omega)\setminus\{0\}}
       \frac{\|u-\pi_{h}u\|\|v-I_{h}v\|}{\|v\|_1} 
 \leq C h \|u-\pi_{h}u\|, 
\end{align*}
where we have used  $\|v-I_h v\| \leq C h \|v\|_1$ (see \cite[(4.6)]{SZ}).
Moreover, 
\begin{align*}
 \|u-\pi_h u\| \leq \|u-I_h u\| \leq c h \|u\|_1, 
\end{align*}
which combined with the previous estimate leads to \eqref{equ:neg}.
\end{proof}
\begin{theorem}\label{thm:approx}
Let $u\in U$ and $y=Y(u) \in V\cap \h^2(\Omega)$ \textnormal{(}so that $\nu> \mathcal{M}(y)$\textnormal{)}, 
and $L$, $M$ be the velocity/pressure 
operators  of the linearized Navier-Stokes equations about $y$, and $L_h$, $M_h$ their discrete counterparts. 
 There exists constants $C$, $C_1= C_1(y)$, $C_2=C_2(y)$, 
 and $C_3=C_3(y)$ such that the following hold:
\begin{enumerate}
\item[\textnormal{(a)}]
 smoothing:\vspace{-5pt}
 \begin{eqnarray}
   \label{equ:lnssm}
   &&\|L v\|\leq C_1 \|v\|_{\tih^{-2}(\Omega)} \quad 
  \forall v \in \Lb^2(\Omega),\\ 
  \label{equ:lnsp}
  &&\|M v \|\leq C_2 \|v\|_{\tih^{-1}(\Omega)}\quad \forall v \in \Lb^2(\Omega).
 \end{eqnarray}
\item[\textnormal{(b)}]
approximation: \vspace{-5pt}
 \begin{eqnarray}
   \label{equ:nsfe}
   &&\|Y (u) - Y_h (u)\| \leq C h^2 \|u\| \quad \forall u \in U,\\
   \label{equ:lnsfeh1}
   && 
      \|L v-L_h v\|_1 \leq C_1 h\|v\|  \quad \forall v \in \mathbf{L}^2(\Omega),
      \\
   \label{equ:lnsfe}
   &&\|L v - L_h v\|\leq C_1 h^2 \|v\| 
    \quad \forall v \in \mathbf{L}^2(\Omega),\\
    \label{equ:lnspfe}
    &&\|M v - M_h v\|\leq C_2 h \|v\| 
	    \quad \forall v \in \mathbf{L}^2(\Omega),\\
    \label{equ:alnsfeh1}
    && 
     \|L^*v - L_h^* v\|_1 \leq C_3 h\|v\| 
     \quad \forall v \in \mathbf{L}^2(\Omega),
     \\
    \label{equ:alnsfe}
    &&\|L^*v - L_h^* v\| \leq C_3 h^2 \|v\| 
   \quad \forall v \in \mathbf{L}^2(\Omega),
 \end{eqnarray}
\item[\textnormal{(c)}]
stability:\vspace{-5pt}
\begin{eqnarray}
  \label{equ:nssta}
  &&\|Y_h (u) \| \leq C \|u \| \quad \forall u \in U,\\
  \label{equ:lnsstab}
   &&\|L_h v \| \leq C_1 \|v\| \quad \forall v \in \mathbf{L}^2(\Omega),\\
  \label{equ:lnspstab}
   &&  \|M_h v \| \leq C_2 \|v\| \quad \forall v \in \mathbf{L}^2(\Omega),\\
    \label{equ:alnsstab}
   &&\|L^*_h v \| \leq C_3 \|v\| \quad \forall v \in \mathbf{L}^2(\Omega).
\end{eqnarray}
\end{enumerate}
\end{theorem}
\begin{proof}
The statement at (a) is similar to the case of the Stokes problem~\cite{DS}. For~\eqref{equ:nsfe}
in~(b) see \cite{Gb1}, page 32, and for \eqref{equ:lnsfeh1}--\eqref{equ:alnsfe} see \cite{GHS}.
The stability in~(c) follows from \eqref{equ:ystab}, (a), and (b). 
\junk{ 
The $L^{\infty}$ estimates in~(d)
are related to results in~\cite{GNS}, where it is shown that 
 the following result holds for the Stokes problem:
 \begin{align*}
  \|\nabla y_h \|_{L^{\infty}(\Omega)} + 
  \|p_h\|_{L^{\infty}(\Omega)} \leq 
  C (\|\nabla y \|_{L^{\infty}(\Omega)} + \|p\|_{L^{\infty}(\Omega)} ),
 \end{align*}
 with $C$ independent of $h,y,p$.
 It is also known that in $\R^2$ 
 \begin{align*}
  W^{2,r}(\Omega) \subset W^{1,\infty}(\Omega), \text{ for } r> 2.
 \end{align*}
 The solution of the Stokes problem belongs to 
 $W^{2,r}(\Omega)^2\times W^{1,r}(\Omega)$, whenever the forcing term $u$ belongs to 
  $L^r(\Omega)^2$ for some real $r>2$, when $\Omega$ is convex and $r$ depends 
  on the largest inner angle of $\partial \Omega$. It is shown in \cite{BB} 
  that for our case the solution is in $W^{2,2+\epsilon}(\Omega)$.
}
\junk{\begin{enumerate}[(a)]
\item 
Similar to the Stokes problem, see \cite{DS}. 
We have 
\begin{align*}
\|Lv\|^2 = |(Lv,Lv)|= |(v,L^*Lv)| \leq 
\|v\|_{\h^{-2}(\Omega)}\|L^*(Lv)\|_{\h^2(\Omega)}\leq
c \|v\|_{\h^2(\Omega)} \|Lv\|
\end{align*}
which implies \eqref{equ:lnssm}.
\item For \eqref{equ:nsfe} see \cite{Gb1}, page 32, and for 
 \eqref{equ:lnsfe}--\eqref{equ:alnsfe} see \cite{GHS}.
 \item Follows from \eqref{equ:ystab}, (a) and (b).
 \item 
}
\end{proof}
\begin{rem}
Theorem~\ref{thm:approx} and Lemma~\ref{lemma:proj} imply that there is 
a constant $C>0$ independent of $h$ such that 
\begin{align}\label{equ:negl}
 \|L(I-\pi_h)v\|\leq C h^2\|v\| \quad \forall v\in \Lb^2(\Omega)
\end{align}
and 
\begin{align}\label{equ:nppl}
\|M(I-\pi_h)v\|\leq C h\|v\| \quad \forall v\in \Lb^2(\Omega).
\end{align}
\end{rem}

For a polygonal domain $\Omega \subset \R^2$, the weighted Sobolev space 
$W_0^{1,0}(\Omega)$ is defined to be the class of functions for which the 
following norm is finite:
\begin{align*}
 \|w\|^2_{W_0^{1,0}(\Omega)} = \int_{\Omega}|\nabla w|^2 dx + 
 \int_{\Omega} \delta(x)^{-2} |w|^2dx,
\end{align*}
where
\(
 \delta(x) = \min\{\textrm{dist}(x,P): P 
                         \textrm{ a vertex of } \Omega\}.
\)
\textcolor{black}{The following regularity and approximation result plays {an} important role in the analysis from Section~\ref{sec:twogridsection}.}
\begin{theorem}
\label{thm:thm8}
 Let $\Omega\subset \R^2$ be a convex polygonal domain, $u \in U$, 
 $y=Y(u)\in V$  and 
 $f \in \mathbf{L}^2(\Omega)$, $g \in W_0^{1,0}(\Omega)$, 
$\int_{\Omega} g dx = 0$. Furthermore, let   
  $\zt = \lt (f,g)$, 
 $\tilde \rho =\tilde M (f,g)$ 
be the weak solution of 
\begin{equation}\label{equ:adj_nzd}
\begin{aligned}
-\nu \Delta \zt - (y \cdot \nabla) \zt + (\nabla y)^T \zt + 
 \nabla \tilde \rho 
         &= f  \quad \text{in } \Omega,\\
 \di \zt &= g  \quad \text{in } \Omega,\\
     \zt &= 0 \quad \text{on }  \partial \Omega.
\end{aligned}
\end{equation}
Then $\zt \in \h^2(\Omega)$, $\tilde \rho \in H^1(\Omega)$ and there exists 
a constant $C = C(\Omega,y)>0$ such that
\begin{align}\label{equ:snzreg}
 \|\zt\|_{\h^2(\Omega)} + 
 \|\nabla  \tilde \rho\| \leq C(\|f\|_{\mathbf{L}^2(\Omega)} + 
 \|g\|_{W_0^{1,0}(\Omega)}).
\end{align}
Moreover, if $\zt_h$ 
is the velocity  of the corresponding discrete problem, then
\begin{align}\label{equ:fenz}
 \|\zt -\zt_h\|_1 \leq C h (\|f\|_{\mathbf{L}_2(\Omega)} + 
  \|g\|_{W_0^{1,0}(\Omega)}), 
 \quad 
 \|\zt_h\|_1 \leq   C(\|f\|_{\mathbf{L}_2(\Omega)} + 
  \|g\|_{W_0^{1,0}(\Omega)}).
\end{align}
\end{theorem}
\begin{proof}
The existence of a unique solution 
$(\zt, \tilde \rho) \in X \times Q$ of \eqref{equ:adj_nzd} and the estimate 
\begin{align} \label{equ:snz}
 \|\zt\|_{\h^1(\Omega)} +\|\tilde \rho\|\leq C(\|f\|_{-1} + \|g\|), \, 
\end{align}
follow from standard results for 
saddle point problems \cite{Br}.
In \cite{KO}, it is shown that under the hypotheses of the theorem, 
the solution of the generalized Stokes system  
\begin{align*}
 -\nu \Delta z + \nabla \rho &= f \quad \text{in } \Omega, \\
 \di z &= g \quad \text{in } \Omega, \\
     z &= 0 \quad \text{on } \partial \Omega,
\end{align*}
satisfies $z \in \h^2(\Omega)$, $\rho \in H^1(\Omega)$ and
\begin{align*}
 \|z\|_{\h^2(\Omega)} + \|\nabla \rho\| \leq C(\|f\| + \|g\|_{\textcolor{black}{W_0^{1,0}(\Omega)}}). 
\end{align*}
Using this result together with \eqref{equ:snz}, it is straightforward to show 
\eqref{equ:snzreg} using the same approach as in Lemma~\ref{lemma:reglin}.
For finite element spaces $X_h, Q_h$ that satisfy the inf-sup condition, we have
\begin{align*}
 \|\zt - \zt_h\|_{\h^1(\Omega)} + \|\tilde \rho-\rho_h\| \leq
C (\inf_{\phi_h\in X_h} \|\zt-\phi_h\|_{\h^1(\Omega)} + 
  \inf_{q_h \in Q_h} \|\tilde \rho - q_h\|),
\end{align*} 
which combined with interpolation estimates yields \eqref{equ:fenz}.
\end{proof}
\junk{
\begin{rem}
 We will need some maximum norm estimates for the solution 
 of the linearized Navier-Stokes equations. In \cite{GNS}, it is shown that 
 the following result holds for the Stokes problem:
 \begin{align*}
  \|\nabla y_h \|_{L^{\infty}(\Omega)} + 
  \|p_h\|_{L^{\infty}(\Omega)} \leq 
  C (\|\nabla y \|_{L^{\infty}(\Omega)} + \|p\|_{L^{\infty}(\Omega)} ),
 \end{align*}
 with $C$ independent of $h,y,p$.
 
 It is also known that in $\R^2$ 
 \begin{align*}
  W^{2,r}(\Omega) \subset W^{1,\infty}(\Omega), \text{ for } r> 2.
 \end{align*}
 The solution of the Stokes problem belongs to 
 $W^{2,r}(\Omega)^2\times W^{1,r}(\Omega)$, whenever $f$ belongs to 
  $L^r(\Omega)^2$ for some real $r>2$, when $\Omega$ is convex and $r$ depends 
  on the largest inner angle of $\partial \Omega$. It is shown in \cite{BB} 
  that for our case the solution is in $W^{2,2+\epsilon}(\Omega)$.
\end{rem}
}

\section{Two-grid preconditioner}
\label{sec:twogridsection}
\textcolor{black}{
In this section we present the construction of the two-grid preconditioners for the 
velocity control and mixed/pressure control problems, and their analyses. 
The main results of this  paper are 
Theorems~\ref{thm:hessv} and~\ref{thm:hessvp} and their
Corollaries~\ref{cor:velmg} and~\ref{cor:velpressmg}}.
We begin with the description of the discrete 
\textcolor{black}{Hessians for the two problems} in 
Section~\ref{ssec:discHessian}, followed by the construction and analysis of
the two-grid preconditioners in Section~\ref{ssec:twogrid}. The velocity control
and mixed/pressure control are treated separately, as the form of the Hessian 
differs significantly in the two cases.

\subsection{The discrete Hessian}
\label{ssec:discHessian}
The discrete Hessian operator at $u \in U \cap X_h$
is defined by the equality 
\begin{equation}
\label{eq:discHessgeneral}
(H_{\beta}^h(u)v,g)=\jh_h''(u)[v,g],\quad \forall v, g \in  X_h.
\end{equation}

\subsubsection{Velocity control}
As in the continuous case,  when $\gamma_p = 0$ we have 
\begin{align}\label{equ:gradd}
 \nabla \jh_h(u) = Y_h'(u)^*(Y_h(u)-y_d^h) + \beta u, \quad u \in U \cap X_h,
\end{align}
with the second variation of the discrete cost functional 
being given by  
\begin{align}\label{var2}
 \jh_h''(u)[g_1,g_2] &= (Y_h'(u)g_2,Y_h'(u)g_1) + 
 (Y_h(u) -y_d^h, Y_h''(u)[g_2,g_1]) +
                   \beta(g_1,g_2). 
\end{align}
The second variation $\lambda_h = Y''_h(u)[g_1,g_2]\in X_h^0$ is 
the solution of
\begin{equation}\label{equ:snsd}
\begin{aligned}
 a(\lambda_h, \phi_h) &+\ct(y_h;\lambda_h,\phi_h) + 
    \ct(\lambda_h;y_h,\phi_h) + b(\phi_h,\mu_h)  \\ 
 &= -\ct(Y_h'(u)g_1;Y_h'(u)g_2,\phi_h)-\ct(Y_h'(u)g_2;Y_h'(u)g_1,\phi_h) 
  \quad \forall \phi_h \in X_h^0, \\
   b(\lambda_h, q_h) &= 0 \quad \forall q_h \in Q_h.
\end{aligned}
\end{equation}
The discrete adjoint $z_h=Y'_h(u)^*(y_h-y_d^h)=L_h^*(Y_h(u)-y_d^h)$ 
is the solution of 
\begin{equation}\label{equ:alnsd}
\begin{aligned}
 a(z_h, \phi_h) &+\ct(y_h;\phi_h,z_h) + \ct(\phi_h;y_h,z_h) + b(\phi_h,\rho_h)\\
 &= (y_h-y_d^h,\phi_h) 
  \quad \forall \phi_h \in X_h^0,\\
   b(z_h, q_h) &= 0 \quad \forall q_h \in Q_h.
\end{aligned}
\end{equation}
Using the same approach as in the continuous case, we obtain
\begin{align*}
-\ct(L_h g_1; L_h g_2, z_h) -\ct(L_h g_2; L_h g_1,z_h) = 
 (y_h-y_d^h,\lambda_h)
\end{align*}
and 
\begin{align*}
\jh_h''(u)[g_1,g_2] = (L_h g_1,L_h g_2) 
-\ct(L_h g_1; L_h g_2, z_h) -\ct(L_h g_2; L_h g_1,z_h)
 + \beta(g_1,g_2). 
\end{align*}
Hence, the discrete Hessian is given by 
\begin{align}\label{equ:dhess}
H_{\beta}^h(u)v = \beta v + L_h^*L_h v + \C_h(u)v
= \beta v + A_h(u) v + \C_h(u)v,
\end{align}
where
\begin{align}\label{equ:bh}
 (\C_h(u)v,v) = -2\ct(L_hv; L_hv, z_h).
\end{align}


\subsubsection{Mixed/pressure control}
Similarly with the derivation in Section~\ref{ssec:Mixed_pressure_control},
in the case of mixed/pressure control, the discrete Hessian takes the form
\begin{align}\label{equ:dhesspress}
 H_{\beta}^h(u)v = \beta v + \gamma_y L_h^*L_h + \gamma_p M_h^*M_h + 
                   \tc_h(u)v,
\end{align}
where 
\begin{align}\label{equ:bhpress}
(\tc_h(u)v,v) = -2\ct(L_h v;L_h v, \zt_h)
\end{align}
 and $\zt_h$ is the solution 
of the discrete 
problem \eqref{equ:adj}.


\subsection{Two-grid preconditioner for discrete Hessian}
\label{ssec:twogrid}
In this section, we construct and analyze a two-grid  preconditioner for the
discrete Hessian $H_{\beta}^h(u)$ defined in~\eqref{equ:dhess} and~\eqref{equ:dhesspress}. 
The construction is a natural extension of the technique used for the optimal control of the Stokes
equations in~\cite{DS}, and is the same for both velocity- and mixed/pressure control.
Let $X_h = X_{2h}\oplus W_{2h}$ be the $L^2$-orthogonal decomposition, where
we consider on $X_h$ the Hilbert-space structure inherited from $\mathbf{L}^2(\Omega)$. 
Let $\pi_{2h}$ be the $L^2$-projector onto $X_{2h}$. 
For $u \in U \cap X_h$ we define the two-grid preconditioner 
\begin{equation}\label{equ:prec}
 T_{\beta}^h(u) = H_{\beta}^{2h}(\pi_{2h} u) \pi_{2h}+ \beta(I-\pi_{2h}).
\end{equation}
It is worth noting that 
\begin{equation}\label{equ:precinv}
 (T_{\beta}^h(u))^{-1} = (H_{\beta}^{2h}(\pi_{2h} u))^{-1} \pi_{2h}+ \beta^{-1}(I-\pi_{2h}).
\end{equation}
We should remark that the  difference between the preconditioner in~\eqref{equ:prec} and the one in~\cite{DS}
is given by the dependence of the Hessian on the control $u$, which forces us to choose a coarse-level
control $u_c\in X_{2h}$ at which the coarse Hessian $H_{\beta}^{2h}(u_c)$ in~\eqref{equ:prec}
is computed. The natural choice is $u_c=\pi_{2h} u$.

\subsubsection{
Analysis for the case of velocity control}
\textcolor{black}{To assess the quality of the preconditioner we use the spectral distance
between $H_{\beta}^h(u)$ and $T_{\beta}^h(u)$
defined in~\cite{DD} for two symmetric positive definite operators
$T_1, T_2 \in \Lo(V_h)$ as
\begin{align}\label{equ:sd}
 d_h(T_1,T_2) = \underset{w \in V_h \setminus \{0\}}{\sup}
                \Bigg|\ln \frac{(T_1 w,w)}{(T_2 w,w)}\Bigg|.
\end{align}
Recall that, cf.~\eqref{equ:dhess} and~\eqref{equ:prec},
\begin{equation}\label{equ:precvel}
T_{\beta}^h(u)= (\beta I + A_{2h}(\pi_{2h} u) + \C_{2h}(\pi_{2h} u)) \pi_{2h} +
 \beta(I-\pi_{2h}).
\end{equation}
The key result is the following.
\begin{theorem}\label{thm:hessv}
 Given $u\in U\cap X_h$, there exists a constant
 $C=C(\Omega,u,{y_d})$ 
 such that
 \begin{align}\label{equ:mainv}
  \|(H_{\beta}^h(u)-T_{\beta}^h(u))v\| \leq C h^2\|v\| 
  \quad \forall v \in X_h.
 \end{align}
\end{theorem}
It is noteworthy that the estimate in Theorem~\ref{thm:hessv} is symmetric in the sense that the same norm (namely the $\mathbf{L}^2$-norm) appears 
on both sides of~\eqref{equ:mainv}, and that the estimate is of optimal order with respect to $h$. This enables us to prove the following result.
\begin{cor}
\label{cor:velmg}
Let $u \in U \cap X_h$. If $\mathcal{C}_h(u)$ is symmetric 
positive definite then 
\begin{align}\label{equ:sdv}
 d(H_{\beta}^h(u),T_{\beta}^h(u)) \leq \frac{C}{\beta}h^2,
\end{align}
for $h < h_0(\beta, \Omega, L)$.
\end{cor}
\begin{proof} By Theorem~\ref{thm:hessv},
\begin{align*}
 \Bigg|\frac{(T_{\beta}^h(u)v,v)}{(H_{\beta}^h(u)v,v)} -1 \Bigg|
 &\leq \frac{C}{\beta} \frac{ h^2 \|v\|^2}
              {\|v\|^2 + 
               \beta^{-1}(\|L_hv\|^2 + (\mathcal{C}_h(u)v,v))} 
  \leq   \frac{C}{\beta}h^2.
\end{align*}
Assume $C \beta^{-1}h_0^2 = \alpha <1$,
and $0<h\leq h_0$. Hence $T_{\beta}^h(u)$ is positive definite
and
\begin{align*}
  \underset{v \in X_h\setminus\{0\}}{\sup} 
  \Bigg|\ln \frac{(T_{\beta}^h(u)v,v)}{(H_{\beta}^h(u)v,v)} \Bigg|
 &
  \leq \frac{|\ln(1-\alpha)|}{\alpha}
  \underset{v \in X_h \setminus \{0\}}
  \sup \Bigg|\frac{(T_{\beta}^h(u) v,v)}{(H_{\beta}^h(u) v,v)} \Bigg|\\
 &\leq \frac{|\ln(1-\alpha)|}{\alpha}
       \frac{C}{\beta} h^2 \leq \frac{C}{\beta} h^2, 
\end{align*}
where we also used that for $\alpha \in (0,1)$,
$x \in [1-\alpha,1+\alpha]$ we have
\begin{align*}
 \frac{\ln(1+\alpha)}{\alpha} |1-x| \leq |\ln x| \leq
 \frac{|\ln(1-\alpha)|}{\alpha}|1-x|.
\end{align*}
\end{proof}
}

\textcolor{black}{Prior to presenting the proof of Theorem~\ref{thm:hessv}
we prove some preliminary lemmas.} 
\begin{lemma}
Let $u\in U \cap X_h$ and $y = Y(u)$, $p=P(u)$, $\pt=P(\pi_{2h}u)$, $\yb= Y(\pi_{2h}u)$. 
Also, let $v \in X_h$ and $w=L(u)v$, $q=M(u)v$, $\wb=L(\pi_{2h}u)v$, 
$\bar q =M(\pi_{2h} u)v$. 
Then there exists 
a constant $K=K(u,\nu,\Omega)>0$ such that
\begin{align}\label{equ:esty}
& |y-\yb|_1 \leq K\|u - \pi_{2h}u\|_{\tih^{-1}(\Omega)},\\
\label{equ:pb}
& \|p-\pt\|\leq K h^2\|u\|_1,
\end{align}
and a constant $C$ independent of $h$ such that
\begin{align}\label{equ:estl}
& \|w-\wb\|_1\leq C h^2 \|u\|_1\|v\|,\\
\label{equ:estp}
& \|q-\bar q\|\leq C h^2 \|u\|_1\|v\|.
\end{align}
\end{lemma}
\begin{proof}
Since $y$ and $\yb$ are the solutions of the Navier-Stokes equations 
with forcing~$u$, $\pi_{2h}u$, respectively, we have  
\begin{align*}
 &a(y,\phi) + c(y;y,\phi) = (u,\phi)\quad \forall \phi \in V,\\
 &a(\yb, \phi) + c(\yb;\yb,\phi) = (\pi_{2h}u,\phi) \quad \forall \phi \in V.
\end{align*}
By taking $\phi=y-\yb$ and subtracting the equations we obtain
\begin{align*}
a(y-\yb,y-\yb) + c(y-\yb; y,y-\yb) + c(\yb;y-\yb,y-\yb) = 
(u-\pi_{2h}u, y-\yb).\end{align*}
Given that $c(\yb;y-\yb,y-\yb) =0$, we obtain
\begin{eqnarray*}
\nu|y-\yb|^2_1&= &(u-\pi_{2h}u,y-\yb)-c(y-\yb;y,y-\yb)\\
&\leq& \|u-\pi_{2h}u\|_{\tih^{-1}}\| y-\yb\|_1 + \m(y)|y-\yb|_1^2.
\end{eqnarray*}
Since \textcolor{black}{$\m(y)< \nu$} and $y, \yb\in X=\h_0^1(\Omega)$, we get 
\begin{eqnarray*}
(\nu - \m(y)) |y-\yb|^2_1 \leq \|u-\pi_{2h}u\|_{\tih^{-1}}\| y-\yb\|_1 \le C \|u-\pi_{2h}u\|_{\tih^{-1}}|y-\yb|_1,
\end{eqnarray*}
which implies~\eqref{equ:esty}. 
From the weak formulations of the Navier-Stokes equations in $X$, with 
forcing $u$, $\pi_{2h}u$ respectively, we have
\begin{align*}
b(\phi,p-\pt) = (u-\pi_{2h}{u},\phi)-a(y-\yb,\phi)+c(\yb;\yb,\phi) 
-c(y;y,\phi). 
\end{align*}
Thus  for $\phi \in X$
\begin{align*}
|b(\phi,p-\pt)|&\leq \|u-\pi_{2h}u\|_{-1}\|\phi\|_1 + \nu\|y-\yb\|_1\|\phi\|_1 +
|c(\yb;y-\yb,\phi)+c(\yb-y;y,\phi)|\\
&\leq \|u-\pi_{2h}u\|_{-1}\|\phi\|_1 + \nu \|y-\yb\|_1\|\phi\|_1  + 
\|y-\yb\|_1\|\phi\|_1 (\|\yb\|_1+ \|y\|_1).
\end{align*}
Then, from the inf-sup condition 
\begin{align}
\label{equ:infsup}
 \beta^{*} \|q-\bar q\| \leq \sup_{0\neq \phi \in X} 
 \frac{|b(q-\bar q,\phi)|}{\|\nabla \phi\|},  
\end{align}
combined with \eqref{equ:esty}, \eqref{equ:neg}, we obtain
\begin{align*}
\|p-\pt\| \leq C(\nu,u,\beta^*)h^2\|u\|_1.
\end{align*}
Recall that $(w,q)$ (resp.~$(\wb,\qb)$) satisfy the linearized 
Navier-Stokes equations~\eqref{equ:lin} about~$y$ (resp.~$\yb$) with 
with forcing $v$, whose weak form in $V$ read:
\junk{
\begin{equation}\label{equ1}
\begin{aligned}
 -\nu \Delta w + (w\cdot \nabla) y + (y \cdot \nabla) w + \nabla q &= v,
  \quad \text{in } \Omega \\
  \di w &= 0 \quad \text{in } \Omega,\\
  w &= 0, \quad \text{on } \partial \Omega,
\end{aligned}
\end{equation}
and $(\wb,\qb)$ is the solution of the linearized Navier-Stokes equations 
about $\yb$
\begin{equation}\label{equ2}
\begin{aligned}
-\nu \Delta \wb + (\wb \cdot \nabla)\yb + (\yb \cdot \nabla)\wb + \nabla
\bar{q} 
         &=v, \quad \text{in } \Omega  \\
 \di \wb &= 0, \quad \text{in } \Omega \\
      \wb &= 0, \quad \text{on } \partial \Omega.
\end{aligned}
\end{equation}
} 
\begin{align}
  \label{equ1}
 a(w,\phi) +c(w;y,\phi)+c(y;w,\phi) = (v, \phi) \quad \forall \phi \in V,\\
 \label{equ2}
 a(\wb, \phi) + c(\wb; \yb, \phi) + c(\yb;\wb, \phi) = (v, \phi) 
 \quad \forall \phi \in V.
\end{align}
By taking $\phi = w-\wb$ in the equations above and subtracting we obtain
\begin{equation}\label{equ:aw}
\begin{aligned}
 -a(w-\wb,w-\wb) = &  c(w; y;w-\wb)+c(y;w,w-\wb)\\ 
 &-c(\wb;\yb,w-\wb)-c(\yb;\wb,w-\wb).
\end{aligned}
\end{equation}
We have
\begin{align*}
 c(w;y, w-\wb) -c(\wb;\yb,w-\wb) &= c(w;y-\yb,w-\wb) + c(w-\wb;\yb,w-\wb)\\
c(y;w,w-\wb) - c(\yb;\wb;w-\wb) &=c(y-\yb; w, w-\wb), 
\end{align*}
where we used $c(\yb;w-\wb,w-\wb)=0$ (see Lemma~\ref{lemma:trilinear}).
Using these in \eqref{equ:aw}, we obtain
\begin{align*}
\nu |w-\wb|_1^2 = |c(y-\yb;w;w-\wb) + c(w;y-\yb, w-\wb) 
 +c(w-\wb;\yb;w-\wb)|.
\end{align*}
From the continuity of the trilinear form $c$ and \eqref{equ:ell} we get 
\begin{align*}
 \nu |w-\wb|^2_1 \leq \mathcal{M}\left(|y-\yb|_1 |w|_1|w-\wb|_1 +
                                 |w|_1|y-\yb|_1|w-\wb|_1 \right)
 +\m(\yb)|w-\wb|^2_1
\end{align*}
which leads to 
\begin{align*}
 (\nu -\m(\yb))|w-\wb|^2_1 \leq 2 \mathcal{M} |w|_1|y-\yb|_1|w-\wb|_1.
\end{align*}
Since $\|\pi_{2h}u\|\le \|u\|$, $\pi_{2h}u\in U$, and so $\nu - \m(\yb)>0$; hence 
we obtain
\begin{align}\label{equ:w}
 |w-\wb|_1 \leq C |y-\yb|_1 |w|_1  
 \overset{\eqref{equ:estlinh1},\eqref{equ:neg},\eqref{equ:esty}}{\leq} 
  C h^2 \|u\|_1\|v\|.
\end{align}
with $C$ depending on $\nu$, $y$, $\kappa$, $\mathcal{M}$, but not on $h$.
\junk{ 
Finally, from Lemma~\ref{lemma:proj}, we have
\begin{align*} 
 \|u-\pi_{2h}u\|_{H^{-1}(\Omega)} \leq C_3 h^2 \|u\|_{H^1(\Omega)}, 
\end{align*}
which combined with \eqref{equ:w}, \eqref{equ:u} yields \eqref{equ:estl}. 
}
To prove \eqref{equ:estp}, we consider the weak formulations of 
\eqref{equ1} and \eqref{equ2} in $X$
\begin{align*}
 a(w,\phi) + c(w;y,\phi) + c(y;w,\phi) + b(q,\phi) &= (v,\phi) \quad 
                                                     \forall \phi \in X, \\
 a(\wb,\phi) + c(\wb;\yb,\phi) + c(\yb;\wb,\phi) + b(\qb,\phi) &= (v,\phi) 
 \quad \forall \phi \in X,
\end{align*}
from which we obtain
\begin{align*}
 b(q-\bar q,\phi) 
 &= -a(w-\wb,\phi)-c(w;y,\phi)-c(y;w,\phi)+c(\wb;\yb,\phi)+c(\yb;\wb,\phi)\\
 &= -a(w-\wb,\phi)-c(w;y-\yb,\phi)-c(w-\wb;\yb,\phi)-c(y;w-\wb,\phi) \\
 &\qquad\hspace*{18mm}   -c(y-\yb;\wb,\phi), \quad \forall \phi \in X.
\end{align*}
Thus, $\forall \phi \in X$
\begin{align*}
|b(q-\qb,\phi)|\leq C |\phi|_1 
\left(|w-\wb|_1+ |w-\wb|_1 (|y|_1 + |\yb|_1)  + |y-\yb|_1(|w|_1+|\wb|_1)\right).
\end{align*}
Using the inf-sup condition~\eqref{equ:infsup}
we obtain
\begin{align*}
 \|q-\qb\| &\leq C(|w-\wb|_1+|y-\yb|_1(|w|_1 + |\wb|_1) + 
  |w-\wb|_1(|y|_1 + |\yb|_1))\\
&\overset{\eqref{equ:w}}\leq 
 C |y-\yb|_1(|\wb|_1+ |w|_1(1 + |y|_1 + |\yb|_1))\\
&\leq C |y-\yb|_1 \|v\|
 \overset{\eqref{equ:neg},\eqref{equ:esty}}\leq 
 C h^2 \|u\|_1\|v\|.
\end{align*}
\end{proof}
\begin{lemma}
Let $u\in U \cap X_h$ and $y=Y(u)$, $\yb=Y(\pi_{2h}u)$.
Also, let $v\in X_h$ and $z = L^*(y-y_d)$, $\zb=L^*(\yb-y_d)$. 
Then there exists a constant $C=C(u,y_d)$ independent of $h$ such that 
\begin{align}\label{equ:zthm}
 \|z-\zb\| \leq Ch^2\|u\|_1.
\end{align}
\end{lemma}
\begin{proof}
 Recall that $z$ and $\zb$ are solutions of 
\begin{align*}
 &a(z,\phi)+c(y;\phi,z)+c(\phi;y,z)= (y-y_d,\phi) 
 \quad \forall \phi \in V\\
 &a(\zb,\phi)+c(\yb;\phi,\zb)+c(\phi;\yb,\zb)=(\yb-y_d,\phi)
 \quad \forall \phi \in V.
\end{align*}
By taking $\phi=z-\zb$ in the previous equations and subtracting them we obtain
\begin{align*}
\lefteqn{ \nu|z-\zb|_1^2 }\\
 &\leq |(y-\yb,z-\zb)|+ 
  |c(y-\yb;z-\zb,\zb)| + |c(z-\zb;y-\yb,z)| +|c(z-\zb;\yb,z-\zb)|\\
 &\leq C_1\|y-\yb\|_{\h^{-1}}|z-\zb|_1 + \|y-\yb\|_1\|z-\zb\|_1(\|\zb\|_1 + \|z\|_1)
   +\mathcal{M}(\yb)|z-\zb|_1^2,
\end{align*}
which gives
\begin{align*}
 (\nu-\mathcal{M}(\yb))|z-\zb|_1^2 
 &\leq |z-\zb|_1(C_1\|y-\yb\| + \|y-\yb\|_1(\|z\|_1 +\|\zb\|_1)).
\end{align*}
Hence, 
\begin{align*}
 |z-\zb|_1 &\leq  \|y-\yb\|_1(C_1 + C_2\|y-y_d\|+C_3\|\yb-y_d\|) 
\overset{\eqref{equ:esty},\eqref{equ:neg}}\leq C(u,y_d)h^2 \|u\|_1
\end{align*}
from which \eqref{equ:zthm} follows immediately. 
\end{proof}
\begin{lemma}
 Let $u \in U\cap X_h$, $y=Y(u)$, $y_h=Y_h(u)$. Also, let $z=L^*(y-y_d)$ and 
 $z_h=L^*_h(y_h-y_d^h)$. Then there exists
$ C=C(u,y_d)$  independent of $h$ so that
 \begin{align}\label{equ:ydestimate}
  \|y_h-y_d^h\|&\leq C (\|u\| + \|y_d\|_1), \\
\label{equ:zestimate} 
  \|z - z_h\|_k& \leq C h^{2-k}\|u\|,\ \ k=0,1.
 \end{align}
\end{lemma}
\begin{proof}
We have
\begin{align*}
 \|y_h-y^h_d\| 
 &\leq \|y_h\|+\|y_d\|+\|y_d-y_d^h\|
 \overset{\eqref{equ:nssta}}{\leq} C \|u\| + \|y_d\| + \|y_d-y_d^h\|\\
 &\leq C\|u\| +\|y_d\|+\|y_d -I_hy_d\| 
  \leq C(\|u\| + \|y_d\|+ h\|y_d\|_1).
\end{align*}
For $h<1$ this leads to \eqref{equ:ydestimate}. 
To prove \eqref{equ:zestimate}, recall that $z$ and $z_h$ satisfy~\eqref{ad} and~\eqref{equ:alnsd},
respectively.
Let $(\bar{z}_h,\bar{\rho}_h)$ be the solution of 
\begin{equation}\label{bzh_eq}
 \begin{aligned}
  a(\bar{z}_h,\phi_h)&+\ct(y_h;\phi_h,\bar{z}_h)+\ct(\phi_h;y_h,\bar{z}_h) +b(\phi_h,\bar{\rho}_h)\\
  &= (y-y_d,\phi_h) \quad \forall \phi_h \in X_h^0,\\
                b(\bar{z}_h,q_h) &= 0 \quad \forall q_h \in Q_h.
 \end{aligned}
 \end{equation}
From~\eqref{equ:alnsfeh1}-\eqref{equ:alnsfe}, we have 
 \begin{align}\label{z_zh}
   \|z-\bar{z}_h\|_k\leq C h^{2-k} \|y-y_d\|, \ \ k=0, 1.
 \end{align}
By taking $\phi_h=z_h-\bar{z}_h$ in~\eqref{equ:alnsd} and~\eqref{bzh_eq}
and subtracting the equations we obtain 
\begin{align*}
 \nu|z_h-\bar{z}_h|_1^2 &+ \ct(y_h;z_h-\bar{z}_h, z_h-\bar{z}_h) +
 \ct(z_h-\bar{z}_h;y_h,z_h-\bar{z}_h) + b(z_h-\bar{z}_h,\rho_h-\bar{\rho}_h)\\
 &=(y-y_h,z_h-\bar{z}_h) - (y_d-y_d^h,z_h-\bar{z}_h),
\end{align*}
which, by using \eqref{equ:ctilde} and $(y_d-y_d^h,z_h-\bar{z}_h) =0$, simplifies to 
\begin{align*}
 \nu |z_h-\bar{z}_h|_1^2 + \ct(z_h-\bar{z}_h;y_h,z_h-\bar{z}_h) = 
 (y-y_h,z_h-\bar{z}_h).
\end{align*}
Thus, 
\begin{align*}
 \nu|z_h-\bar{z}_h|^2_1 \leq \|y-y_h\|\|z_h-\bar{z}_h\| + 
 \m(y_h)|z_h-\bar{z}_h|_1^2.
\end{align*} 
Since $\nu -\m(y_h)>0$, we obtain
\begin{align*}
 \|z_h-\bar{z}_h\|_1 \leq C \|y-y_h\| \stackrel{\eqref{equ:nsfe}}{\leq} C h^2\|u\|,
\end{align*}
which combined with \eqref{z_zh} proves the lemma. 
\end{proof}

\textcolor{black}{We now present the proof of Theorem~\ref{thm:hessv}.}
\begin{proof}
Cf.~\eqref{equ:dhess} and~\eqref{equ:precvel},
 \begin{align*}
  T_{\beta}^h(u) - H_{\beta}^h(u) = A_{2h} (\pi_{2h}u) \pi_{2h} - A_h(u) + 
  \C_{2h}(\pi_{2h}u)\pi_{2h} - \C_h(u).
 \end{align*}
We first estimate 
\begin{equation}\label{equ:da}
\begin{aligned}
 A_{2h}(\pi_{2h}u)\pi_{2h} - A_h(u) 
 &= [A_{2h}(\pi_{2h}u)- A(\pi_{2h}u)]\pi_{2h} 
 + A(\pi_{2h}u) (\pi_{2h}-I) \\
 &+ A(\pi_{2h}u) - A(u) + A(u) - A_h(u).
\end{aligned}
\end{equation}
For any $v\in X_h$ we have
\begin{align*}
    |(A(u)-A_h(u))v,v)| 
 &= |(L^*Lv-L_h^*L_hv,v)| = |\|Lv||^2 -\|L_hv\|^2| \\
 &\leq \|(L-L_h)v\|(\|Lv\|+\|L_hv\|) \leq C h^2 \|v\|^2,
\end{align*}
which implies 
\begin{align*}
 \|(A(u)-A_h(u))v\| \leq C h^2 \|v\|,
\end{align*}
since $A(u)-A_h(u)$ is symmetric on $X_h$. 
Similarly, it can be shown that
\begin{align*}
 \|(A_{2h}(\pi_{2h}u)- A(\pi_{2h}u))\pi_{2h}v\| \leq C h^2\|v\|.
\end{align*}
For the second term in \eqref{equ:da} we have
\begin{align*}
 \|A(\pi_{2h}u) (\pi_{2h}-I) v\| = 
 \|L^*(\pi_{2h}u)L(\pi_{2h}u)(\pi_{2h}-I)v\| 
 \overset{\eqref{equ:negl}}{\leq} Ch^2\|v\|.
\end{align*}
Finally, we have 
\begin{align*}
   |&(A(\pi_{2h}u)v - A(u)v, v)| 
 = |(L^*(\pi_{2h}u)L(\pi_{2h}u)v - L^*(u)L(u)v,v)|\\ 
& = |\|L(\pi_{2h}u)v\|^2-\|L(u)v\|^2| 
  \leq \|(L(\pi_{2h} u)v - L(u)v\|(\|L(\pi_{2h}u)v\| + \|L(u)v\|)\\
& \overset{\eqref{equ:estl}}{\leq}Ch^2\|u\|_1\|v\|,
\end{align*}
which implies
\[
 \| (A(\pi_{2h}u) - A(u))v\| \leq C h^2 \|v\|.
\]
Combining this with the previous estimates we obtain
\begin{align}\label{equ:term3}
 \|(A_{2h}(\pi_{2h}u) - A_h(u))v\| \leq C h^2 \|v\|.
\end{align}
Next, we estimate 
\begin{equation}\label{equ:db}
\begin{aligned}
 \C_{2h}(\pi_{2h}u)\pi_{2h} - \C_h(u) 
 &= (\C_{2h}(\pi_{2h}u) - \C(\pi_{2h}u))\pi_{2h} 
+ \C(\pi_{2h}u)(\pi_{2h}-I) \\
&\quad + \C(\pi_{2h}u) - \C(u) 
+ \C(u)- \C_h(u).
\end{aligned}
\end{equation}
We begin by estimating the term  
 $\|\C(u)v - \C_h(u)v\| $.  
Let $y=Y(u)$, $y_h=Y_h(u)$, $z=L^*(y-y_d)$, $z_h=L^*_h(y_h-y_d^h)$.
Cf.~\eqref{equ:cform} and~\eqref{equ:bh},
\begin{align*} 
 (\C(u)v,v) = -2c(Lv;Lv,z) \text{ and } 
 (\C_h(u)v,v) = -2\ct(L_h v;L_h v,z_h).
\end{align*}
We have $c(Lv;Lv,z) = \ct(Lv;Lv, z)$ since $L v \in V$. Therefore,
\begin{align*}
|(\C(u)v - \C_h(u)v,v)| &= 
|2\ct(Lv;Lv,z)-2\ct(L_hv;L_hv,z_h)| \\
&\leq |c(Lv;Lv,z)-c(L_hv;L_hv,z_h)|\\
&\quad+ |c(Lv;z,Lv)-c(L_hv;z_h,L_hv)|.
\end{align*}
The first term in the inequality above can be bounded by
\begin{align*}
|c(L_hv&;L_hv,z_h)-c(Lv;Lv,z)| \\
&\leq
|c((L_h-L)v;L_hv,z_h)| + |c(Lv;L_hv, z_h-z)|+|c(Lv;z,(L_h-L)v|, 
\end{align*}
where we used $c(Lv;(L_h-L)v,z) = -c(Lv;z,(L_h-L)v)$, since $Lv \in V$.\\
{
We have
\begin{align*}
\lefteqn{       |c((L_h-L)v;L_hv,z_h)| }\\
 &\leq |c((L_h-L)v;(L_h-L)v,z_h)|+ |c((L_h-L)v;Lv,z_h)|\\ 
 &\leq \|(L_h-L)v||_1 \|(L_h-L)v\|_1\|z_h\|_1 + 
       \|(L_h-L)v\| \|Lv\|_2\|z_h\|_1 \\
 &\overset{\eqref{equ:lnsfeh1},\eqref{equ:lnsfe}}{\leq} C h^2\|v\|^2\|z_h\|_1
  \overset{\eqref{equ:alnsstab}}{\leq} C h^2\|v\|^2\|y_h-y_d^h\| 
 \overset{\eqref{equ:ydestimate}}{\leq}C(u,y_d)h^2\|v\|^2,
\end{align*}
and
\begin{align*}
	|c(Lv;L_hv,z_h-z)| 
 &\leq |c(Lv;(L_h-L)v,z_h-z)| + |c(Lv;Lv,z_h-z)|\\
 &\leq \|Lv\|_1 \|(L_h-L)v\|_1 \|z_h-z\|_1 + C\|Lv\|_1\|Lv\|_2\|z_h-z\|\\
 &\overset{\eqref{equ:lnsfeh1},\eqref{equ:estlinh2},\eqref{equ:zestimate}}{\leq}
 C h^2\|v\|^2\|y_h-y_d^h\| \overset{\eqref{equ:ydestimate}}{\leq }C(u,y_d)h^2\|v\|^2.
\end{align*}
Combining these estimates with
\begin{align*}
 |c(Lv;z,(L_h-L)v)|\leq C\|Lv\|_1\|z\|_2\|(L_h-L)v\| 
 \overset{\eqref{equ:regalin},\eqref{equ:lnsfe}}\leq C(u,y_d)h^2\|v\|^2,
\end{align*}
we obtain
\begin{align*}
|c(L_hv&;L_hv,z_h)-c(Lv;Lv,z)| \leq C(u,y_d)h^2\|v\|^2.
\end{align*}
} 
Similarly, 
\begin{align*}
|c(L_hv&;z_h,L_hv) -c(Lv;z,Lv)| \\
&\leq 
|c((L_h-L)v;z_h,L_hv)| +|c(Lv;z_h,(L_h-L)v)|+|c(Lv;Lv,z-z_h)|\\
&\leq |c((L_h-L)v;z_h-z,L_hv)| + |c((L_h-L)v;z,L_hv)| \\
&\quad + |c(Lv;z_h-z,(L_h-L)v)| + |c(Lv;z,(L_h-L)v)|+|c(Lv;Lv,z-z_h)|\\
&\leq \|(L_h-L)v\|_1 \|z_h-z\|_1\|L_hv\|_1 + 
      C \|(L_h-L)v\|\|z\|_2 \|L_hv\|_1 \\
&\quad + \|Lv\|_1\|z_h-z\|_1\|(L_h-L)v\|_1 + C\|Lv\|_1\|z\|_2 \|(L_h-L)v\|\\
&\quad + C\|Lv\|_1\|Lv\|_2\|z-z_h\|\overset{\eqref{equ:regalin},\eqref{equ:lnsfeh1},\eqref{equ:lnsfe},\eqref{equ:zestimate}}
\leq C(u,y_d)h^2\|v\|^2.
\end{align*}
Using the same approach, it can be shown that 
\[
\|(\C_{2h}(\pi_{2h}u) - \C(\pi_{2h}u)\pi_{2h}v \| \leq C h^2\|v\|.
\] 
Let $\bar{z}=L^*(Y(\pi_{2h} u) -y_d)$.
The third term in \eqref{equ:db} can be bounded as 
\begin{align*}
 |(\C(\pi_{2h} u)v - \C(u)v,v)| &=
 2|c(Lv;Lv,\bar{z}) - c(Lv;Lv,z)|= 2|c(Lv;Lv,\bar{z}-z)|\\ 
 &\leq C \|Lv\|_1 \|Lv\|_2 \|\bar{z}-z\|\stackrel{\eqref{equ:estlinh2}}{\leq} C \|v\|^2 \|\bar{z}-z\|\\
 &\overset{\eqref{equ:zthm}}{\leq}C(u,y_d)h^2\|u\|_1\|v\|^2.
\end{align*}
Finally let $w =(\pi_{2h}-I)v$. With $L=L(\pi_{2h}u)$, we have
\begin{align*}
\lefteqn{ |(\C(\pi_{2h}u)(\pi_{2h}-I)v,v)|}\\
&=  |((L w \cdot \nabla) \bar{z} - (\nabla L w )^T\bar{z}, Lv)|
\leq |((L w \cdot \nabla )\bar{z}),Lv)| + 
     |(\nabla L w)^T\bar{z}, Lv)|\\
&= |c(L w; \bar{z}, Lv)| + |c(Lv;L w, \bar{z})| 
\overset{\eqref{equ:trilinear}}{=} |c(L w; \bar{z}, Lv)| + |c(Lv;\bar{z}, L w)| \\
&\overset{\eqref{equ:trilinear}}{\leq} C \|Lw\| \|\bar{z}\|_2 \|Lv\|_1  \overset{\eqref{equ:negl}}\leq 
C h^2 \|v\| \|\bar{z}\|_2 \|v\| 
\overset{\eqref{equ:regalin}}{\leq} C(u,y_d) h^2 \|v\|^2.
\end{align*}
\end{proof}

\subsubsection{Analysis for the case of mixed/pressure control}
Recall from~\eqref{equ:dhesspress} that in the case of mixed/pressure control, 
the discrete Hessian takes the form
\begin{align*}
 H_{\beta}^h(u)v = \beta v + \gamma_y A_h v+ \gamma_p B_h v + 
                   \tc_h(u)v,
\end{align*}
where $B_h= M_h^*M_h$ and $A_h=L_h^*L_h$ as in~\eqref{equ:dhess}.
Following the definition in~\eqref{equ:prec}, the two-grid preconditioner takes the form
\begin{equation}\label{equ:hdmp}
\begin{aligned}
    T_{\beta}^h(u) 
 &= (\beta I + \gamma_y A_{2h}(\pi_{2h}u) + 
	 \gamma_p B_{2h}(\pi_{2h}u) + \tc_{2h}(\pi_{2h}u))\pi_{2h} +
	 \beta(I-\pi_{2h}).
\end{aligned}
\end{equation}
{
\begin{lemma}
Let $u\in U\cap X_h$ and $y=Y(u)$, $p=P(u)$, $\yb=Y(\pi_{2h}u)$, 
$\pt=P(\pi_{2h} u)$. Also, let $v \in X_h$ and 
$\zt=\lt(\gamma_y(y-y_d),\gamma_p(p_d-p))$, 
$\hat{z}=\lt(\gamma_y(\yb-y_d),\gamma_p(p_d-\pt))$, with $\lt$ defined in Theorem~\textnormal{\ref{thm:thm8}}. 
Then there exists a constant $C=C(u,y_d,p_d,\gamma_y,\gamma_p)$ independent of $h$ such that 
\begin{align}\label{equ:ztildethm}
 \|\zt -\hat{z}\|_1\leq C h \|u\|_1^{1/2}.
\end{align}
\begin{proof}
Recall that $(\zt,\rt)$ is the solution of~\eqref{equ:adj}, and $(\hat{z},\hat{\rho})$ satisfies
\begin{equation}
\label{equ:adjbar}
\begin{aligned}
 a(\hat{z},\phi) + c(\bar{y};\phi,\hat{z}) + c(\phi;\bar{y},\hat{z}) + b(\phi,\hat{\rho}) 
          &= \gamma_y(\bar{y}-y_d,\phi) \quad \forall \phi \in X, \\
 b(\hat{z},q) &= \gamma_p(\bar{p}-p_d,q) \quad \forall q \in Q,
\end{aligned}
\end{equation}
By subtracting~\eqref{equ:adjbar} from~\eqref{equ:adj} we obtain
\begin{eqnarray*}
a(\zt-\hat{z},\phi)+c(y;\phi,\zt-\hat{z})&+&c(\phi;y,\zt-\hat{z}) + b(\phi,\rt-\bar{\rho})=\\
&&\gamma_y(y-\bar{y})+c(\bar{y}-y;\phi,\hat{z})+c(\phi;\bar{y}-y,\hat{z})\\
b(\zt-\bar{z},q)&=&\gamma_p(\bar{p}-p),\ \ \forall  \phi \in X, q \in Q,
\end{eqnarray*}
which represents the weak form of
\begin{align*}
-\nu\Delta(\zt-\hat{z}) -(y\cdot\nabla)(\zt-\hat{z}) + 
 (\nabla y)^T(\zt-\hat{z})
 +\nabla(\rt- \hat{\rho}) &= 
 \gamma_y(y-\yb) 
  + ((y-\yb)\cdot\nabla)\hat{z} \\
 &-(\nabla(y-\yb))^T\hat{z} \quad \text{in } \Omega,\\
 \di(\rt -\hat{\rho})& = \gamma_p(\pt-p) \quad \text{in } \Omega\\
 \zt -\hat{z}&= 0 \quad \text{on } \partial \Omega.
\end{align*}
Using \eqref{equ:snz}, we get
\begin{align}\label{equ:z}
 \|\zt-\hat{z}\|_1 &\leq C(\gamma_y \|y-\yb\|+ \gamma_p\|p-\pt\|+
 \|((y-\yb)\cdot\nabla)\hat{z}\| + \|(\nabla(y-\yb))^T\hat{z}\|). 
\end{align} 
From \eqref{equ:h1l4},  \textcolor{black}{combined with 
\eqref{equ:esty},\eqref{equ:neg},\eqref{equ:snz}}, we have 
\begin{align*}
 \|((y-\yb)\cdot\nabla)\hat{z}\| \leq C\|y-\yb\|_1\|\nabla \hat{z}\|
\leq C h^2 (\gamma_y\|\bar{y}-y_d\| + \gamma_p\|\bar{p}-p_d\|).
\end{align*}
Of the four terms in the right hand side of~\eqref{equ:z}, only the last is of order one in $h$:
\begin{align*}
\|(\nabla(y-\yb))^T\hat{z}\|
&\leq \|\nabla(y-\yb)\|_{\Lb^4(\Omega)}\|\hat{z}\|_{\Lb^4(\Omega)}
\leq C \|\nabla(y-\yb)\|^{1/2}\|\nabla\nabla(y-\yb)\|^{1/2}\|\hat{z}\|_1\\
&\overset{\eqref{equ:h2nse},\eqref{equ:esty},{\eqref{equ:neg}},\eqref{equ:snz}}\leq
 C(u) h\|u\|_1^{1/2}(\|\yb-y_d\|+\|\pt-p_d\|).
\end{align*}
Using these estimates  together with
\eqref{equ:esty}, \eqref{equ:pb}, \eqref{equ:neg}  in \eqref{equ:z} we obtain 
\begin{align*}
 \|\zt -\hat{z}\|_1 \leq C(u) h\|u\|_1^{1/2}(\gamma_y\|\bar{y}-y_d\| + \gamma_p\|\bar{p}-p_d\|).
\end{align*}
\end{proof}
\end{lemma}
}
\textcolor{black}{We are now in a position to prove the main result for mixed/pressure control.}
\begin{theorem}\label{thm:hessvp}
Let $u\in U\cap X_h$ be so that $p = P(u) \in W_0^{1,0}(\Omega)$. If 
$p_d \in W_0^{1,0}(\Omega) \cap Q$, then  there exists 
a constant $C=C(\Omega, u,y_d,p_d,\gamma_y,\gamma_p)$  such that
\begin{align*}
 \|(H_{\beta}^h(u)-T_{\beta}^h(u))v\| \leq C h\|v\| \quad \forall v\in X_h.
\end{align*}
\end{theorem}

\begin{proof}
For any $u\in U \cap X_h$, we have 
\begin{align*}
     T_{\beta}^h(u) - H_{\beta}^h(u) 
 &= \gamma_y(A_{2h} (\pi_{2h}u) \pi_{2h} - A_h(u)) + 
    \gamma_p(B_{2h} (\pi_{2h}u) \pi_{2h} - B_h(u)) \\
 &+ \tc_{2h}(\pi_{2h}u)\pi_{2h} - \tc_h(u).
\end{align*}
We use the same approach as in the case of velocity control only. 
We have already shown in~\eqref{equ:term3} that the first term is
$O(h^2\|v\|)$. The second term is estimated similarly:
\begin{equation}\label{equ:dpa}
\begin{aligned}
     B_{2h}(\pi_{2h}u)\pi_{2h} - B_h(u) 
 &= [B_{2h}(\pi_{2h}u) - B(\pi_{2h}u)]\pi_{2h} 
 + B(\pi_{2h}u) (\pi_{2h}-I) \\
 &+ B(\pi_{2h}u) - B(u)  + B(u) - B_h(u).
\end{aligned}
\end{equation}
For any $v \in X_h$ we have
\begin{align*}
    |((B(u)-B_h(u))v,v)| 
 &= |(M^*M v-M_h^*M_h v,v)| = 
    |\|M v\|^2-\|M_h v\|^2| \\
 &\leq \|M v - M_h v\|(\|M v\| +\|M_h v\|) \leq C h\|v\|^2,  
\end{align*}
where we used \eqref{equ:lnspfe} and \eqref{equ:lnspstab}. 
Similarly, it can be shown 
\begin{align*}
 \|(B_{2h}(\pi_{2h}u)-B(\pi_{2h}u))v\|\leq C h \|v\|^2, \quad \forall v\in X_h. 
\end{align*}
The second term in \eqref{equ:dpa} can be bounded as
\begin{align}
 \label{equ:Btermpressest}
 \|B(\pi_{2h}u)(\pi_{2h}-I)v\| =
 \|M^*(\pi_{2h}u)M(\pi_{2h}u)(\pi_{2h}-I)v \| 
 \overset{\eqref{equ:nppl}}{\leq} C h \|v\|. 
\end{align}
Finally, we have
\begin{align*}
  |(B(\pi_{2h}u)v &- B(u)v, v)| 
 = |(M^*(\pi_{2h}u)M(\pi_{2h}u)v - M^*(u)M(u)v,v)| \\
 &= |\|M(\pi_{2h}u)v\|^2-\|M(u)v\|^2| \\
 & \leq \|(M(\pi_{2h} u)v - M(u)v\|(\|M(\pi_{2h}u)v\| + \|M(u)v\|)
 \overset{\eqref{equ:estp}}{\leq}Ch^2\|u\|_1\|v\|,
\end{align*}
which gives 
\begin{align}
  \|(B(\pi_{2h}u) - B(u))v\| \leq Ch^2\|u\|_1\|v\|.
\end{align}
Next, we estimate 
\begin{equation}\label{equ:dbm}
\begin{aligned}
    \tc_{2h}(\pi_{2h}u)\pi_{2h} - \tc_h(u) 
&= [\tc_{2h}(\pi_{2h}u) - \tc(\pi_{2h}u)]\pi_{2h}  
 + \tc(\pi_{2h}u)(\pi_{2h}-I) \\
& + \tc(\pi_{2h}u) - \tc(u)+ \tc(u)- \tc_h(u).
\end{aligned}
\end{equation}
We first estimate the term  
 $\|\tc(u)v - \tc_h(u)v\| $,  
and recall that 
\begin{align*} 
 (\tc(u)v,v) = -2c(Lv;Lv,\zt) \text{ and } 
 (\tc_h(u)v,v) = -2\ct(L_h v;L_h v,\zt_h), 
\end{align*}
with $\zt=\lt(\gamma_y(y-y_d),\gamma_p(p_d-p))$, 
$\zt_h=\lt_h(\gamma_y(y_h-y_d^h),\gamma_p(p_d^h-p_h))$, with the operators
$\lt$ and $\lt_h$ as defined in Theorem~\ref{thm:thm8}.
Thus, 
\begin{align*}
|(\tc(u)v - \tc_h(u)v,v)| &= 
|2\ct(Lv;Lv,\zt)-2\ct(L_hv;L_hv,\zt_h)| \\
&\leq |c(Lv;Lv,\zt)-c(L_hv;L_hv,\zt_h)|\\
&\quad+ |c(Lv;\zt,Lv)-c(L_hv;\zt_h,L_hv)|.
\end{align*}
The first term in the inequality above can be bounded by
\begin{align*}
|c(L_hv;L_hv,\zt_h)-c(Lv;Lv,\zt)| &\leq
|c((L_h-L)v;L_hv,\zt_h)| + |c(Lv;L_hv, \zt_h-\zt)|\\
&\quad +|c(Lv;\zt,(L_h-L)v|, 
\end{align*}
where we used $c(Lv;L_hv-Lv,\zt) = -c(Lv;\zt,L_hv-Lv)$
since $Lv \in V$.
Thus, we have 
\begin{align*}
  \lefteqn{|c(L_hv;L_hv,\zt_h)-c(Lv;Lv,\zt)| }\\
 &\leq    C(\|Lv -L_hv\|_1 \|L_hv\|_1 \|\zt_h\|_1 + \|Lv\|_1 \|L_hv\|_1 \|\zt_h-\zt\|_1 
  + \|Lv\|_1 \|\zt\|_1 \|L_hv-Lv\|_1 )\\
 & \stackrel{\eqref{equ:lnsfeh1},\eqref{equ:fenz}}{\leq} 
           C h\|v\|^2(\gamma_y\|y-y_d\| + \gamma_p\|p-p_d\|_{W_0^{1,0}(\Omega)}).
\end{align*}
Note that we have used~\eqref{equ:fenz} also for 
$\|L_h v\|_1\leq C \|v\|$, since $L_h v = \lt_h(v,0)$. Also,
\begin{align*}
  |c(L_hv&;\zt_h,L_hv) -c(Lv;\zt,Lv)| 
\leq |c((L_h-L)v;\zt_h,L_hv)| +|c(Lv;\zt_h,(L_h-L)v)| \\
&+|c(Lv;Lv,\zt-\zt_h)|\\
&\leq \|L_h v -Lv\|_1 \|\zt_h\|_1 \|L_h v\|_1 
+\|Lv\|_1 \|\zt_h\|_1 \|L_hv - Lv\|_1 \\
&+\|Lv\|_1\|Lv\|_1 \|\zt-\zt_h\|_1 
\leq Ch\|v\|(\gamma_y\|y-y_d\| + \gamma_p\|p-p_d\|_{W_0^{1,0}(\Omega)}).
\end{align*}
Similarly, it can be shown that 
$\|(\tc_{2h}(\pi_{2h}u) - \tc(\pi_{2h}u)\pi_{2h}v \| \leq C h\|v\|$. 
To estimate the third term in \eqref{equ:dbm}, let 
$\hat{z}=\lt(\gamma_y(Y(\pi_{2h}u)-y_d),\gamma_p(p_d-P(\pi_{2h}u)))$.
Then
\begin{align*}
 &|(\tc(\pi_{2h} u)v - \tc(u)v,v)| =
 2|c(Lv;Lv,\zt) - c(Lv;Lv,\hat{z})|= 2|c(Lv;Lv,\zt-\hat{z})|\\ 
 &\leq C \|Lv\|_1^2 \|\zt-\hat{z}\|_1 \stackrel{\eqref{equ:fenz}}{\leq} 
C \|v\|^2 \|\zt-\hat{z}\|_1\overset{\eqref{equ:ztildethm}}\leq C(u,y_d,p_d,\gamma_y,\gamma_p) h\|v\|^2.
\end{align*}
Finally, let $w=(\pi_{2h}-I)v$. With $L=L(\pi_{2h}u)$ we have (see~\eqref{eq:ctilde})
\begin{align*}
  & |(\tc(\pi_{2h}u)(\pi_{2h}-I)v,v)| 
 = |((Lw\cdot \nabla)\hat{z} -(\nabla Lw)^T\hat{z},Lv)| \\
 &\leq |((L w \cdot \nabla )\hat{z}),Lv)| + 
     |(\nabla L w)^T\hat{z}, Lv)| = |c(L w; \hat{z}, Lv)| + |c(Lv;L w, \hat{z})| \\
&= |c(L w; \hat{z}, Lv)| + |c(Lv;\hat{z}, L w)| \leq C \|Lw\| \|\hat{z}\|_2 \|Lv\|  \\
&\overset{\eqref{equ:negl}}\leq 
C h^2 \|v\| \|\hat{z}\|_2 \|v\| 
\overset{\eqref{equ:adj_nzd}}{\leq} C(u,y_d,p_d,\gamma_y,\gamma_p) h^2 \|v\|^2
\end{align*}
which combined with the other estimates yields the conclusion.
\end{proof}
\textcolor{black}{The following result follows from Theorem~\ref{thm:hessvp} using  arguments similar 
to the ones in the proof of Corollary~\ref{cor:velmg}. Essentially it shows a decline by one unit in the approximation order
of the two-grid preconditioner for the case of mixed/pressure control compared to velocity control. This is consistent with the case of the Stokes-constrained
problem studied in~\cite{DS}.}
\begin{cor}
\label{cor:velpressmg}
\textcolor{black}{Under the conditions of Theorem~\ref{thm:hessvp}}, if $\tilde{\C}_h(u)$ is symmetric 
positive definite then 
\begin{align}\label{equ:sdp}
 d(H_{\beta}^h(u),T_{\beta}^h(u)) \leq \frac{C}{\beta}h,
\end{align}
for $h < h_0(\beta,\Omega,L,M,\tilde{L})$.
\end{cor}
\textcolor{black}
{
\begin{rem}
\label{rem:cor2subopt}
We should note that the hypotheses of Theorem~\ref{thm:hessvp} are quite restrictive, in that we cannot expect 
$p_d$ and $P(u)\in W_0^{1,0}(\Omega)$ in practice; smooth functions in $W_0^{1,0}(\Omega)$ tend to zero near the corners
of the domain and there is no a priori reason for a flow to have a zero-pressure near the corners of the domain.
Nevertheless, this fact does not prevent the preconditioner to work quite well for the 
mixed/pressure control, as shown in Section~\ref{sec:numerics}.
\end{rem}
}
\begin{rem}
\label{rem:multigrid}
The two-grid preconditioner can be extended to a multigrid preconditioner following essentially the same  strategy as 
in\textnormal{~\cite{DS}}, and the analysis is extended in a similar fashion to show that the multigrid preconditioner
satisfies the estimates~\eqref{equ:sdv} and~\eqref{equ:sdp}. Suffice it to say that the correct multigrid preconditioner has a 
$W$-cycle structure, while the  associated $V$-cycle gives suboptimal results; furthemore, 
the coarsest level has to be sufficiently fine in order for the optimal quality to be preserved.
\end{rem}

\section{Numerical results}
\label{sec:numerics}
We present a set of numerical results to showcase the behavior of our multigrid preconditioner in the Newton
iteration of~\eqref{equ:ocd} on $\Omega=(0,1)^2$. 
We consider uniform rectangular grids with mesh sizes $h = 1/32, 1/64, 1/128, 1/256$, 
and we use 
Taylor-Hood \mbox{$\mathbf{Q}_2$−-$\mathbf{Q}_1$} elements for velocity-pressure and $\mathbf{Q}_2$ 
elements for the controls.
The data is given by \mbox{$y_d^h=Y_h(u_h)$}, $p_d^h=P_h(u_h)$, with $u_h$ being the interpolant of the 
target control $u(x,y)=[10^3 (\textnormal{sign}(y-0.9)+1) (y-0.9)^2,0]$ (see Figure~\ref{fig:target}); the velocity field resembles 
one obtained from a lid-driven cavity flow. \textcolor{black}{In Figure~\ref{fig:recovered} we show the optimal control and the recovered velocity and pressure 
profile for the velocity control problem. As can be seen in the picture, if $\gamma_p=0$, the pressure is not recovered.}
The Newton iteration is stopped 
when $\|\nabla \hat{J}_h\|_{\infty}\leq 10^{-10}$. On the coarsest grid at $h=1/32$ we use a zero-initial guess for the Newton solve,
while for subsequent grids we start the iteration using the solution from the coarser problem.
The linear systems at each iteration are solved in two ways: first we use
conjugate gradient preconditioned by the multigrid preconditioner (MGCG) (see Remark~\ref{rem:multigrid}),
with base cases $h_0=1/32$ or $1/64$, depending on necessity.
Second, we solve the same systems using unpreconditioned conjugate gradient (CG). The reduced Hessian
is applied matrix-free  using~\eqref{equ:dhess}--\eqref{equ:dhesspress}. Obviously, the Hessian-vector multiplication (matvec) is the most
expensive operation, as it essentially requires solving the linearized Navier-Stokes system twice. The goal is to show that, as a result
of multigrid preconditioning, the number of matvecs at the highest resolution is relatively low compared to 
the unpreconditioned case.

\begin{figure}[!t]
\begin{center}
        \includegraphics[width=5in]{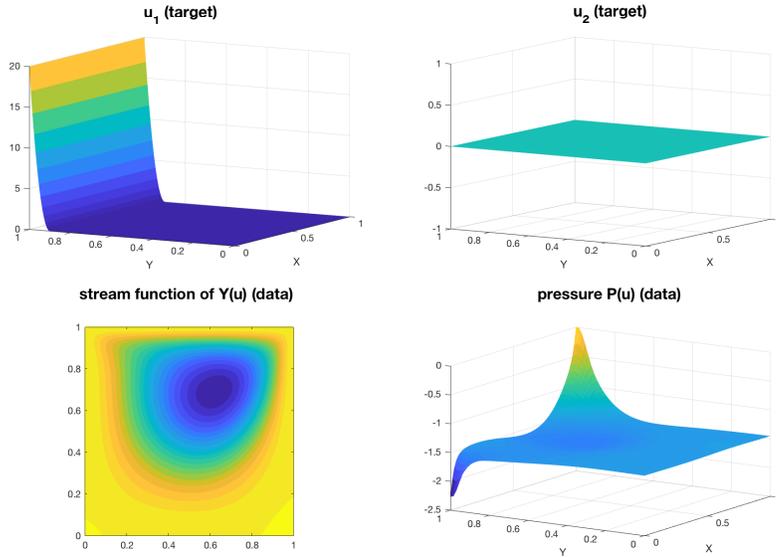}
\caption{Top images: components of target control. Bottom images: velocity (stream function) and pressure data. Viscosity is
$\nu=0.01$ with
$Re =\nu^{-1}\|Y(u)\|_{\infty} \approx 105$, and $h=1/64$.}
\label{fig:target}
\end{center}
\end{figure}

\begin{figure}[!t]
\begin{center}
        \includegraphics[width=5in]{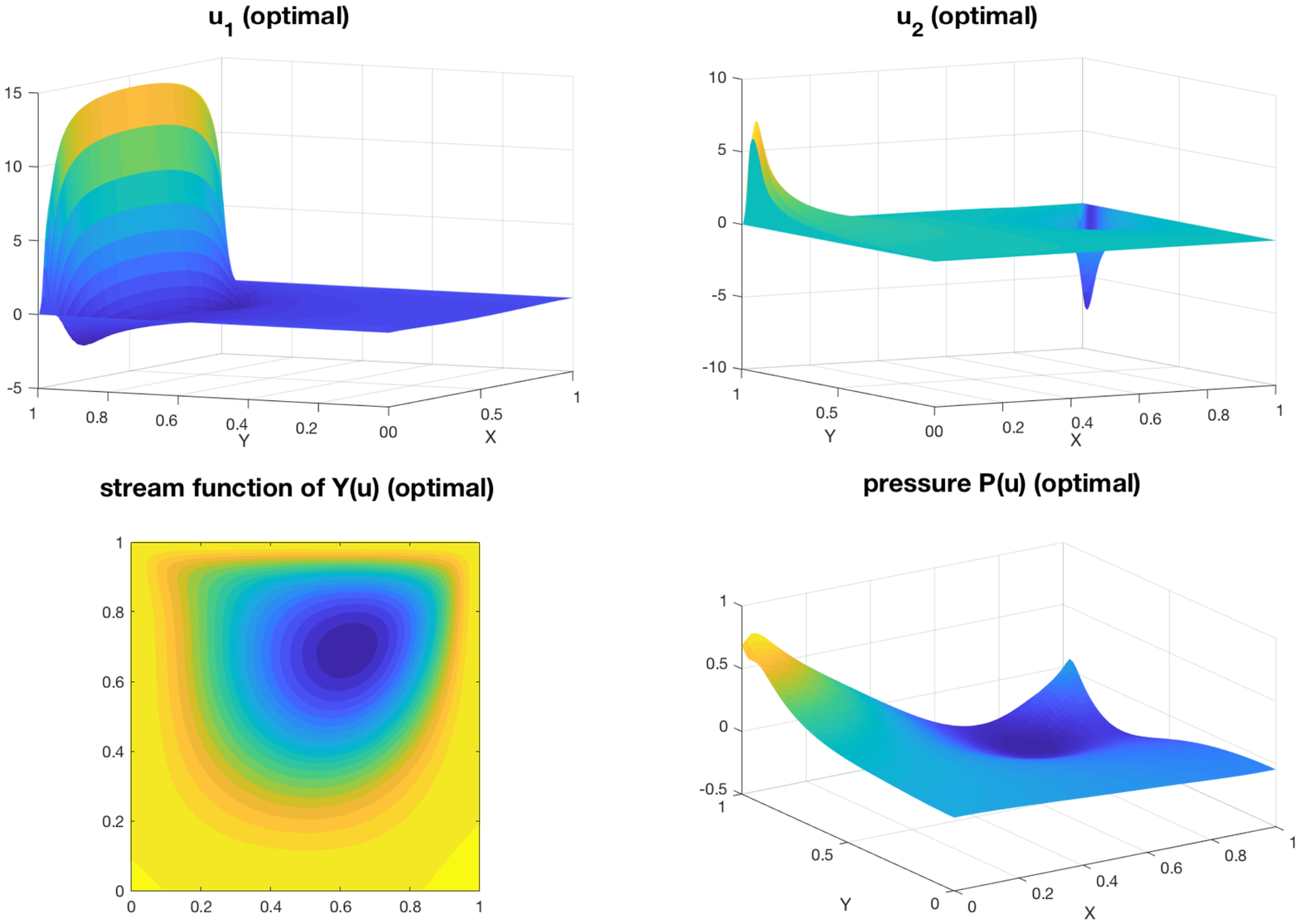}
\caption{Top images: components of optimal control corresponding to data from Figure~\textnormal{\ref{fig:target}}. Bottom images: optimal velocity (stream function) and pressure. 
Parameters values are: $\gamma_y=1$, $\gamma_p=10^{-3}$, and $\beta=10^{-5}$.}
\label{fig:recovered}
\end{center}
\end{figure}

We present in Table~\ref{tab:nu_pm1} results for low  and in Table~\ref{tab:nu_pm2} for moderate 
Reynolds numbers, and we compare
velocity control ($\gamma_p=0$) with  mixed velocity-pressure control with varying ratios of the two terms in $\hat{J}_h$
($\gamma_y=1$, $\gamma_p=10^{-4}, 10^{-3}$). As for the regularization parameter we let
$\beta=10^{-4}, 10^{-5}$.  For each of the twelve parameter choices and for each $h=1/64, 1/128, 1/256$
we report  the number of iterations of the MGCG/CG-based solves for each Newton iteration 
as well as the total (added) wall-clock time of the linear solves \textcolor{black}{(since the overhead due to the gradient computation and Hessian setup is the same
for both solves)}.  
For example, in the top left compartment of 
Table~\ref{tab:nu_pm2}, we show the case $\nu=0.01$, $\gamma_y=1,\gamma_p=0$ (velocity control only), $\beta=10^{-4}$ with 
resolutions $1/64, 1/128, 1/256$. At resolution $h=1/256$ two Newton iterations were required with CG
necessitating $382$ and $274$ iterations, with a total time of linear solves of $11.4$ hours, while the four-grid 
MGCG needed $6$ and $4$ iterations for a total of $0.58$ hours, meaning almost twenty times faster. Note that at the coarsest level
we actually build the Hessian at each Newton iteration and invert it using direct methods, the time of this operation being 
included in the reported wall-clock time. \textcolor{black}{We should note that using direct methods at the coarsest level
is critical for mixed/pressure control, but can be replaced  with low-tolerance iterative solves for velocity control, a phenomenon 
that is still under scrutiny. For the experiments in Tables~\ref{tab:nu_pm1} and~\ref{tab:nu_pm2}} 
the relative tolerance for both CG and MGCG is set at $10^{-8}$.
\textcolor{black}{While this is quite low for the Newton-CG method, it forces a slightly higher number of linear iterations,
thus allowing to better observe the desired behavior of the MGCG preconditioner, namely that for fixed $\beta$ the number of linear iterations
will decrease with $h\downarrow 0$. For two of the twelve cases we also vary the tolerance for the linear solves ($10^{-2}$ and $10^{-4}$) and show the results in 
Table~\ref{tab:nu_pm1tole-2}. Here we also report the total wall-clock time for the computation (including the gradient computation and Hessian setup at 
each Newton iteration),
since  the number of Newton iterations begins to vary between CG to MGCG when using a large tolerance. This
is due mainly to two factors related to MGCG: either it converges very fast and the gradients are better approximated than expected 
even if the tolerance is set at $10^{-2}$, or the MG preconditioner is not positive definite and so MGCG fails to converge (we still report 1  MGCG iteration). 
We allow a maximum of 10 Newton iterations.}

\textcolor{black}{Tables~\ref{tab:nu_pm1} and~\ref{tab:nu_pm2}} indicate a behavior that is standard for the multigrid preconditioner presented in this work, and which is 
consistent with the analysis.
First we notice that unpreconditioned CG is scalable, in the sense that for each case the number of CG iterations is bounded with 
respect to mesh-size (the 
wall-clock times suffer due to the fact that we used direct solvers for the linearized Navier-Stoles solves in the matvec). 
The MGCG instead shows an efficiency
that increases over CG with decreasing $h$, measured both in terms of number of iterations and wall-clock time,
and this can be seen for all the velocity control cases,  and for the mixed control cases with base case $h_0=1/64$ at $\nu=0.1$ 
(see Table~\ref{tab:nu_pm1}). As usual with these types of algorithms, the lower order of approximation for the mixed/pressure control
case leads not only to a slightly higher number of MGCG iterations, but also requires a finer base case; for all the 
mixed velocity-pressure control problems, the four-grid preconditioner at resolution $h=1/256$ (base case $h_0=1/32$)
led to an iteration \textcolor{black}{that does not converge within 10 Newton iterations since the MG-preconditioner is not positive definite}.
However, the base case choice $h_0=1/64$ appears to be sufficient when $\nu=0.1$. For the higher 
Reynolds number case, while we did not encounter divergence with base case $h_0=1/64$, it is conceivable that it may still 
be too coarse, that is, it will lead to divergence at higher resolutions. 
We should point out that we purposefully selected a set of parameters that exhibit a variety of behaviors expected from these
types of algorithms.
Yet we find it remarkable that whenever MGCG  converges, it does so significantly faster 
than unpreconditioned CG, with significant wall-clock savings.

\textcolor{black}{The results in Table~\ref{tab:nu_pm1tole-2} are also instructive. First they show that for the two cases ($\gamma_p=10^{-3}$ and $\beta=10^{-4}, 10^{-5}$),
a tolerance of $10^{-4}$ leads to the same number of Newton iterations as with $10^{-8}$, as in the last group in Table~\ref{tab:nu_pm1}.
Second, it shows that a large tolerance may lead to an increase in the number of Newton iterations coupled with fewer CG/MGCG iterations per Newton iteration, as expected.
However, in almost ever case, this results in a higher wall-clock time for the entire computation due to the costly overhead at the beginning of each Newton iteration.
To conclude, it is preferable to set the tolerance sufficiently low  in order to minimize the number of Newton iterations; then one should use the 
coarsest base case for the MG preconditioner that preserves the good approximation properties of the two-grid preconditioner, as described in detail in~\cite{DS}. Whether MGCG or 
unpreconditioned CG gives a faster wall-clock time certainly depends on the particular problem parameters, but what remains consistent is the 
increasing efficiency of MGCG over CG as $h\downarrow 0$.}

\section{Conclusions \textcolor{black}{and extensions}} 
\label{sec:conclusions}
We have developed and analyzed a two-grid preconditioner to be used in the Newton iteration 
for the optimal control of the stationary Navier-Stokes equations. Under the natural assumption that the iteration
starts sufficiently close to  the solution it is shown that the preconditioner has a behavior that is similar
to the optimal control of the stationary Stokes equations~\cite{DS}. While the extension to multigrid is not
explicitly discussed due to the similarity with the Stokes-control case, numerical results confirm that the behavior 
is consistent with the analysis, and can lead to significant savings over unpreconditioned CG-based solves. 

\textcolor{black}{
The method described in this work extends naturally to boundary control, although the optimal discretization of
the controls and the analysis requires additional fine tuning.
As shown in the case of optimal control of linear elliptic equations, the analysis of the 
analogous MGCG preconditioner for boundary control problems is challenging, and is expected to behave differently than
for distributed optimal control~\cite{monathesis}, with significant differences being observed between 
Dirichlet- and Neumann control. Extending those results to the optimal control of the Stokes and Navier-Stokes
systems forms the subject of our current research. Last, but not least, adding control constraints to a semismooth Newton approach
as in the elliptic control case presents a set of challenges. The technique develped in~\cite{MR3537010} 
for distributed optimal control of linear elliptic equations with control constraints produces a multigrid preconditioner which
approximates the Hessian to $O(h/\beta)$, assuming a piecewise constant discretization of the controls. Thus, 
if applied to Navier-Stokes control, the method 
in~\cite{MR3537010} is expected to yield an optimal-order preconditioner for the mixed/pressure control, but not for velocity
control.  Improving that quality to the optimal order $O(h^2/\beta)$ also forms the subject of
current research.
}

\section*{Acknowledgments} 
The authors thank the editor and the anonymous referees for the careful reading of the manuscript and useful suggestions.

\begin{table}[!h]
 \begin{center}
 \caption{Iteration counts and runtimes for MGCG vs. CG; $\nu=0.1$, $Re =\nu^{-1}\|Y(u)\|_{\infty} \approx 1.1$.
   Tolerance is set at $10^{-8}$.}
      {\begin{tabular}{|@{}c|c|c|c||c|c|c|}\hline
                & \multicolumn{3}{|c||}{$\beta=10^{-4}$} & \multicolumn{3}{c|}{$\beta=10^{-5}$}\\ \hline
          $h_j^{-1}$ & $64$ & $128$ & $256$ & $64$ & $128$ & $256$  \\\hline\hline
           \multicolumn{7}{|c|}{$\gamma_p=0$ (velocity control only)}\\\hline
          \# cg & 44 & 42  & 41  & 107 & 103  & 103  \\\hline
          $t_{\mathrm{cg}}$ (sec.)  & 113 & 497  & 3662 & 285 & 1284 & 9015 \\\hline
          \# mg  $(h_0=1/32)$ & 3  & 3 & 2 & 5 & 4 & 3\\\hline
          $t_{\mathrm{mg}}$ (sec.) & 56 & 149 & 637 & 49 & 166 & 735 \\\hline
          \emph{eff}=$t_{\mathrm{cg}}/t_{\mathrm{mg}}$& 2.02 & 3.34 & 5.75 & 5.82 & 7.73 & 12.27\\\hline\hline
            \multicolumn{7}{|c|}{$\gamma_p=10^{-4}$ (some pressure control)}\\\hline
          \# cg & 46, 39  & 46, 40  & 46, 41  & 116, 102 & 116, 103 &  117, 106\\\hline
          $t_{\mathrm{cg}}$ (sec.)  & 227 & 1016 & 7610 & 579 & 2754 & 19483\\\hline
          \# mg  $(h_0=1/32)$ & 5, 5 & 5, 5 & nc & 8, 6 & 8, 6 & nc \\\hline
          $t_{\mathrm{mg}}$ (sec.) & 122 & 348 & -- & 141 & 389 & --\\\hline
          \emph{eff}=$t_{\mathrm{cg}}/t_{\mathrm{mg}}$&  1.86 & 2.92  & --  & 4.11 & 7.08 & -- \\\hline
          \# mg  $(h_0=1/64)$ & n/a  & 4, 4  & 4, 4 & n/a & 5, 5 & 5, 5\\\hline
          $t_{\mathrm{mg}}$ (sec.) &  & 4216 & 4981 &  & 4073 & 5792  \\\hline
          \emph{eff}=$t_{\mathrm{cg}}/t_{\mathrm{mg}}$&   & 0.24 & 1.53  &  & 0.68 & 3.36\\\hline\hline
          \multicolumn{7}{|c|}{$\gamma_p=10^{-3}$ (more pressure control)}\\\hline
          \# cg &  43, 46  &  44, 48   &  46, 48  & 104, 124 & 106, 120 & 109, 123\\\hline
          $t_{\mathrm{cg}}$ (sec.)  & 239 & 1087 & 8329  & 605 & 2679 & 20676 \\\hline
          \# mg  $(h_0=1/32)$ & 5, 7 & 5, 7 & nc  & 8, 9 & 8, 9 & nc\\\hline
          $t_{\mathrm{mg}}$ (sec.) & 129 & 359 & -- &  140 & 419  & -- \\\hline
          \emph{eff}=$t_{\mathrm{cg}}/t_{\mathrm{mg}}$ & 1.85  & 3.03  & --  & 4.32 & 6.39 & -- \\\hline
          \# mg  $(h_0=1/64)$ & n/a & 5, 7 & 5, 7 &  n/a & 7, 8 & 6, 8\\\hline
          $t_{\mathrm{mg}}$ (sec.) &  & 4360 & 5573 &  & 4427 & 5768  \\\hline
          \emph{eff}=$t_{\mathrm{cg}}/t_{\mathrm{mg}}$ &   & 0.25  & 1.49   &  & 0.61 & 3.58  \\\hline\hline
      \end{tabular}}
      \label{tab:nu_pm1}
\end{center}
\end{table}

\begin{table}[!h]
 \begin{center}
 \caption{Iteration counts and runtimes for MGCG vs. CG; $\nu=0.01$, $Re =\nu^{-1}\|Y(u)\|_{\infty} \approx 104$.
   Tolerance is set at $10^{-8}$.}
      {\begin{tabular}{|@{}c|c|c|c||c|c|c|}\hline
                & \multicolumn{3}{|c||}{$\beta=10^{-4}$} & \multicolumn{3}{c|}{$\beta=10^{-5}$}\\ \hline
          $h_j^{-1}$ & $64$ & $128$ & $256$ & $64$ & $128$ & $256$  \\\hline\hline
           \multicolumn{7}{|c|}{$\gamma_p=0$ (velocity control only)}\\\hline
           \# cg &  281, 289  &  259, 247   & 382, 274  & 819, 837,  &  777, 716   &  1136, \\
                 &            &             &           & 459       &             &    816  \\\hline
          $t_{\mathrm{cg}}$ (hours)  & 0.42 & 1.73 & 11.40 & 1.54 & 5.15 & 33.63\\\hline
          \# mg           & 10, 10 & 9, 9 &  6, 4 & 27, 31,& 25, 27 &  16, 11\\
           $(h_0=1/32)$   &        &      &       &  15    &        &        \\\hline
          $t_{\mathrm{mg}}$ (hours) & 0.05 & 0.15 & 0.58 & 0.11 & 0.29 & 0.93  \\\hline
          \emph{eff}=$t_{\mathrm{cg}}/t_{\mathrm{mg}}$&  8.81 & 11.30 &  19.82 & 14.70 & 18.07 &  36.11 \\\hline\hline
            \multicolumn{7}{|c|}{$\gamma_p=10^{-4}$ (some pressure control)}\\\hline
          \# cg & 299, 318 &  309, 306   &  500, 493 & 834, 909,  &  891, 965   &  1644, \\
                &          &             &           & 507        &             &  1354  \\ \hline
          $t_{\mathrm{cg}}$ (hours)  & 0.43 & 2.12 & 16.92 &  1.45 & 6.36 & 50.57\\\hline
          \# mg         & 12, 13 & 13, 13 & nc & 36, 38,& 35, 38 & nc \\
          $(h_0=1/32)$  &        &        &    & 23     &        &     \\\hline
          $t_{\mathrm{mg}}$ (hours) & 0.05 & 0.19 & -- & 0.11 & 0.38 & --\\\hline
          \emph{eff}=$t_{\mathrm{cg}}/t_{\mathrm{mg}}$& 7.94  & 10.99 &  -- & 12.60  & 16.83 &  -- \\\hline
          \# mg        & n/a  & 7, 7  & 9, 8 &  n/a  & 12, 12  &  19, 14 \\
          $(h_0=1/64)$ &      &       &      &       &         &         \\\hline
          $t_{\mathrm{mg}}$ (hours) &  & 1.27 & 1.62 &   & 1.23 &  1.98\\\hline
          \emph{eff}=$t_{\mathrm{cg}}/t_{\mathrm{mg}}$&   & 1.67 & 10.45 & & 5.16 & 25.48 \\\hline\hline
          \multicolumn{7}{|c|}{$\gamma_p=10^{-3}$ (more pressure control)}\\\hline
          \# cg &  294, 332,  &  296, 336,   &  481, 597 & 922, 995, &   802, 1028,   &  1321,    \\
                &  213        &  161         &           & 728       &   561          &  1818, 629\\\hline
          $t_{\mathrm{cg}}$ (hours)  & 0.62 & 2.70 & 18.62 & 1.62 & 8.18 & 64.39 \\\hline
          \# mg        & 14, 16 & 15, 18, & nc &  42, 51, & 41, 57,& nc\\
          $(h_0=1/32)$ & 12     & 11      &    &  39      & 32     & \\\hline
          $t_{\mathrm{mg}}$ (hours) & 0.08 & 0.31 & -- & 0.13 & 0.59 & -- \\\hline
          \emph{eff}=$t_{\mathrm{cg}}/t_{\mathrm{mg}}$ & 7.99  & 8.97  & --& 12.17 & 13.78 & --   \\\hline
          \# mg        & n/a  & 8, 11,& 9, 17, &  n/a &  14, 18,&  16, 23, \\
          $(h_0=1/64)$ &      & 6     & 9      &      &  13     &  16     \\\hline
          $t_{\mathrm{mg}}$ (hours) &  & 1.78 & 2.84 & {} & 1.97 & 3.19 \\\hline
          \emph{eff}=$t_{\mathrm{cg}}/t_{\mathrm{mg}}$ &   &  1.52 & 6.56 &  & 4.16 & 20.22  \\\hline\hline
      \end{tabular}}
      \label{tab:nu_pm2}
\end{center}
\end{table}

\begin{table}[!h]
 \begin{center}
 \textcolor{black}{
\caption{Iteration counts and runtimes for MGCG vs. CG; $\nu=0.1$, $Re =\nu^{-1}\|Y(u)\|_{\infty} \approx 1.1$.}
      {\begin{tabular}{|@{}c|c|c|c||c|c|c|}\hline
                & \multicolumn{3}{|c||}{$\beta=10^{-4}$} & \multicolumn{3}{c|}{$\beta=10^{-5}$}\\ \hline
          $h_j^{-1}$ & $64$ & $128$ & $256$ & $64$ & $128$ & $256$  \\\hline\hline
           \multicolumn{7}{|c|}{$\gamma_p=10^{-3}$,\hspace{10pt}  $tol=10^{-2}$}\\\hline
          \# cg & \textcolor{black}{12, 17,} &  \textcolor{black}{10, 17}  & \textcolor{black}{10, 17} & 16, 39  &  16, 40,&  \textcolor{black}{15, 37,}  \\
               & \textcolor{black}{18, 17} &  \textcolor{black}{19, 17} & \textcolor{black}{19, 17}  & 44, 46 &  45, 47   &  \textcolor{black}{42, 44} \\\hline
          $t_{\mathrm{cg}}$ (sec.)  & \textcolor{black}{141} &  \textcolor{black}{765} & 5878 & 309 & 1817 & \textcolor{black}{12614} \\\hline
          $t_{\mathrm{total}}$ (sec.) & \textcolor{black}{370} & \textcolor{black}{1779} & 10323 & 536 & 2822 & \textcolor{black}{17175}\\\hline
          \# mg $(h_0=1/32)$  & \textcolor{black}{3, 4, 4}  & \textcolor{black}{3, 4, 4} & \textcolor{black}{2, 1, 4,} &  & 1, 4, & \textcolor{black}{nc}\\
                 &   &  &  \textcolor{black}{1, 4}  &  &   5, 5 & \\\hline
          $t_{\mathrm{mg}}$ (sec.) & \textcolor{black}{178}  & \textcolor{black}{469} & \textcolor{black}{3277} &  & 624 &  --\\\hline
          $t_{\mathrm{total}}$ (sec.) & \textcolor{black}{357}  & \textcolor{black}{1250} & \textcolor{black}{8702} &  & 1625 &-- \\\hline
          \# mg  $(h_0=1/64)$ &   & \textcolor{black}{3, 4} & \textcolor{black}{3, 2, 4} & 1, 5, & 1, 4, 2 & \textcolor{black}{1, 4,} \\
                              &   &  &   &  3, 5 &  & \textcolor{black}{3, 4}\\\hline
          $t_{\mathrm{mg}}$ (sec.) &  & \textcolor{black}{4150}  & \textcolor{black}{7771} &  197 & 6295  &  \textcolor{black}{10787}\\\hline
          $t_{\mathrm{total}}$ (sec.) &  & \textcolor{black}{4722} & \textcolor{black}{11316} & 426 &  7070 & \textcolor{black}{15307}\\\hline\hline
          \multicolumn{7}{|c|}{$\gamma_p=10^{-3}$,\hspace{10pt}  $tol=10^{-4}$}\\\hline
          \# cg & \textcolor{black}{22, 29} & \textcolor{black}{21, 28} & 21, 29 & 44, 79 & \textcolor{black}{45, 76} & \textcolor{black}{46, 78}\\\hline
          $t_{\mathrm{cg}}$ (sec.)  & \textcolor{black}{110}  & \textcolor{black}{586}  & 4788 & 260 & \textcolor{black}{1440} & \textcolor{black}{11429} \\\hline
          $t_{\mathrm{total}}$ (sec.) & \textcolor{black}{239} & \textcolor{black}{1154} & 7375 & 392 &\textcolor{black}{2010}  & \textcolor{black}{14025} \\\hline
          \# mg  $(h_0=1/32)$ & \textcolor{black}{4, 5} & \textcolor{black}{4, 5} & \textcolor{black}{4, 1, 5, 5} & 6, 7 & \textcolor{black}{4, 7} & \textcolor{black}{nc} \\\hline
          $t_{\mathrm{mg}}$ (sec.) & \textcolor{black}{117}  & \textcolor{black}{337} & \textcolor{black}{3114} & 124 & \textcolor{black}{344} &  -- \\\hline
          $t_{\mathrm{total}}$ (sec.) & \textcolor{black}{247} & \textcolor{black}{907} & \textcolor{black}{7605} & 257 &\textcolor{black}{907}  & -- \\\hline
          \# mg  $(h_0=1/64)$ & n/a & \textcolor{black}{3, 5}   & \textcolor{black}{3, 5}  &  n/a & \textcolor{black}{4, 6}  & \textcolor{black}{4, 6} \\\hline
          $t_{\mathrm{mg}}$ (sec.) &  & \textcolor{black}{4160} & \textcolor{black}{5308} &  & \textcolor{black}{4497} & \textcolor{black}{5489} \\\hline
          $t_{\mathrm{total}}$ (sec.) &   & \textcolor{black}{4736} & \textcolor{black}{7894} &  & \textcolor{black}{5066} &  \textcolor{black}{8064} \\\hline\hline
      \end{tabular}}
      \label{tab:nu_pm1tole-2}
}
\end{center}
\end{table}

\bibliographystyle{siam}
\bibliography{reference_mg}

\end{document}